\documentclass[11pt]{article}

\usepackage{amsfonts,amssymb,amsopn,amsmath,mathrsfs,theorem,bbm}
\usepackage[round]{natbib}
\usepackage[breaklinks]{hyperref}
\usepackage{graphicx}
\usepackage{verbatim}
\usepackage{authblk}
\usepackage{color}
\usepackage{enumitem}
\usepackage{comment}
\usepackage{mathtools}

\newcounter{thm_counter}[section]
\newtheorem{lemma}[thm_counter]{Lemma}
\newtheorem{prop}[thm_counter]{Proposition}
\newtheorem{theorem}[thm_counter]{Theorem}

\newtheorem{cor}[thm_counter]{Corollary}

\theorembodyfont{\upshape}
\newtheorem{remark}[thm_counter]{Remark}
\newtheorem{defn}[thm_counter]{Definition}

\newcounter{assumptions}
\newenvironment{proof}{\vspace{1ex}\noindent{\textsc{Proof:}}\hspace{0.5em}}{\hfill\qed\vspace{1ex}}

\numberwithin{equation}{section}
\numberwithin{thm_counter}{section}

\DeclareMathOperator{\var}{var}

\DeclareMathOperator{\col}{col}

\def\ra{\Rightarrow} 
\def\to{\rightarrow} 

\def\sw{\subseteq} 
\def\mc{\mathcal} 
\def\mb{\mathbb} 
\def\sc{\setminus} 
\def\p{\partial} 
\def\v{\mathbf} 
\def\E{\mb{E}} 
\def\P{\mb{P}}
\def\R{\mb{R}} 
\def\N{\mb{N}}

\def\~{\sim}
\def\-{\,:\,} 

\def\|{\,|\,} 

\def\wt{\widetilde}
\def\qed{$\blacksquare$}
\def\1{\mathbbm{1}}
\def\cadlag{c\`{a}dl\`{a}g}

\def\l{\left}
\def\r{\right}

\def\IH1{\textit{($IH$)$_1$}}

\allowdisplaybreaks

\begin{document}

\author{Nic Freeman}
\author{Jonathan Jordan}

\affil{School of Mathematics and Statistics, University of Sheffield}

\title{Extensive Condensation in a model of Preferential Attachment with Fitness}
\date{\today}
\maketitle

\begin{abstract}
We introduce a new model of preferential attachment with fitness,
and establish a time reversed duality between the model and a system of branching-coalescing particles.
Using this duality, we give a clear and concise explanation for the condensation phenomenon, in which unusually fit vertices may obtain abnormally high degree: it arises from a growth-extinction dichotomy within the branching part of the dual.

We show further that the condensation is extensive. As the graph grows, unusually fit vertices become, each only for a limited time, neighbouring to a non-vanishing proportion of the current graph.
\end{abstract}


\section{Introduction}

The classical model of preferential attachment is an increasing sequence of random graphs $(\mc{G}_n)$,
beginning from a finite graph $\mc{G}_0$.
To construct $\mc{G}_{n+1}$ from $\mc{G}_n$, a vertex $p_n$ is randomly sampled from $\mc{G}_n$, with the probability of picking each vertex $v$ weighted according to its degree $\deg_n(v)$. Then, a single new node is attached to $p_n$ via a single new edge.

More generally, the new node may be joined via $m$ new edges to $m$ existing nodes, each sampled independently from $\mc{G}_n$, weighted by degree and with replacement.
This model is perhaps the simplest example of a stochastic model in which earlier gains (in the form of higher degree)
confer an advantage towards future growth.
It has been studied extensively and the structure of $\mc{G}_n$ as $n\to\infty$ is well understood; see for example Chapter 8 of \cite{Hofstad2016} and the references therein.

The classical model was generalized by \cite{BianconiBarabasi2001}, with the addition of fitness values for the vertices.
A higher fitness value confers a better chance of attaching to the new incoming vertices. 
More precisely, nodes are assigned i.i.d.~fitness values $F_v\in[0,1]$, and a node $v$ with fitness $F_v$ now carries weight $F_v\deg_n(v)$ (instead of $\deg_n(v)$).
Cases in which $F_v$ has support $[0,1]$ but $\P[F_v=1]=0$ are of particular interest.
In such cases, as the graph grows large, it is possible that the vertices with fitnesses approaching $1$ will capture a macroscopic fraction of the edges -- a phenomenon known as condensation.

Using evidence from numerical simulations \citeauthor{BianconiBarabasi2001} predicted that once their graph became large \textit{`a single node captures a positive proportion of the links'} -- this is known as `extensive' condensation.
\cite{DereichEtAl2017} showed recently that extensive condensation did not, in fact, occur.

A second extension of the classical model, known as preferential attachment with choice, was studied by \cite{MalyshkinPaquette2014}.
Their model does not include fitnesses; rather, to obtain $\mc{G}_{n+1}$ from $\mc{G}_n$ a set $\{p_1,\ldots,p_R\}$ of vertices are sampled from $\mc{G}_n$, each using the same degree-weighted mechanism as in classic preferential attachment (independently, and with replacement).
A single new vertex then attaches via a single new edge to whichever
$p_i$ has the highest degree.

\citeauthor{MalyshkinPaquette2014} showed that in their model a so-called persistent hub emerges -- a single vertex $v$ which, at some random time $N$, has maximal degree (within $\mc{G}_N$) and which then remains as the vertex of maximal degree for all time. 
This property is key to their arguments. 
When $R>2$, they establish extensive condensation through showing that the degree of the persistent hub grows linearly.

In the present article we introduce a new model, which modifies the model of \cite{MalyshkinPaquette2014} to include fitnesses.
Like \cite{BianconiBarabasi2001}, we take the vertex fitnesses to be i.i.d.~values in $[0,1]$.
In our model, to obtain $\mc{G}_{n+1}$ from $\mc{G}_n$, we sample vertices $\{p_1,\ldots,p_R\}$ from $\mc{G}_n$, each using the same degree-weighted mechanism as in classic preferential attachment (independently, and with replacement).
Then, attach a single new vertex $v_n$ to whichever
$p_i$ has the highest fitness.

\medskip

We will show that, in contrast to the Bianconi-Barab\'{a}si model, in our model extensive condensation \textit{does} occur.
However, it occurs without the emergence of a persistent hub. This results in a delicate situation in which a succession of ever fitter vertices grow to eventually topple the previously dominant positions of older (and less fit) vertices.
Our model provides the first rigorous example of a preferential attachment graph with extensive condensation via such behaviour.
From hereon, let us refer to the model as PAC -- `Preferential Attachment with Choice by fitness'.

We analyse PAC using techniques which, to our knowledge, are novel to preferential attachment;
we exhibit a time-reversed duality between PAC and a system of branching-coalescing particles.
This type of duality is perhaps best known in the context of population genetics
where genealogical trees, described by branching-coalescing particles, are used to represent historical transfers of genetic information.

Note that sampling the vertex $v$ weighted according to $\deg_n(v)$ is equivalent to sampling a half-edge (in $\mc{G}_n$) uniformly at random, and then picking the associated vertex $v$.
For this reason it is advantageous to consider half-edges.
For convenience, we assign to each half-edge the same fitness as its associated vertex.
We will use genealogies to track new half-edges inheriting fitness values from pre-existing half-edges.
These genealogies will be closely connected to the duality used, in a spatial model of population genetics, by \cite{EtheridgeFreemanEtAl2017}.

In the genealogical dual process of PAC, if we suppress coalescence and consider the behaviour when the graph is large,
we obtain that the branching part of the dual approximates a Galton-Watson process,
at least when restricted to only finitely many generations.
Using this fact we will be able to give a clear and concise explanation of why (and, under what condition) condensation occurs: precisely, when this Galton-Watson process has positive probability of non-extinction.
Non-extinction corresponds to the genealogy of a new half-edge extending far backwards in time,
far enough that it has a chance of being descended from an unusually fit vertex born long ago.
Moreover, we will give an intuitive description of the limiting degree-weighted fitness measure, in terms of the number of leaves of a Galton-Watson tree.

The most involved part of the present article will be showing that condensation, when it occurs, is extensive.
Here, we require a more sensitive analysis of the dual process than the Galton-Watson approximation can provide.
We use a mixture of martingale-like calculations and weak convergence techniques,
which will permit us to observe the genealogy of a new half-edge in detail.
We are able to identify an explicit constant $\beta\in(0,1)$ such that
the fittest vertex present at time $\approx n^\beta$
will grow to neighbour an asymptotically positive proportion of $\mc{G}_n$.

\subsection{Condensation in random graphs}
\label{sec:background}

In physics, Bose-Einstein condensation is a phenomenon in which, within particular types of matter and at low temperature, a positive proportion of particles occupy the lowest quantum energy state. Such particles are known as the `condensate' and,
remarkably, their existence permits quantum effects to become visible at macroscopic scale.
Within random graphs, the term condensation was introduced by \cite{BianconiBarabasi2001}, who represented half-edges of their graph as particles within a Bose gas,
with fitnesses corresponding to energy states
-- but inverted, so as the fitness value $1$ corresponds to the zero energy state.

Phase transitions, such as that characterising the emergence of a Bose-Einstein condensate, only become sharp when the number of particles tends to infinity.
However, in this limit, there are \textit{two} natural ways in which one might define what is meant by the emergence of a Bose-Einstein condensate.
Firstly, we might ask that a macroscopic fraction of particles remain in the lowest energy state;
alternatively, we might ask that a macroscopic fraction of particles become arbitrary close to the lowest energy state.
The former definition corresponds to extensive condensation, the latter to non-extensive condensation.

More generally, condensation refers to the formation of an atom in the limit of sequence of measures.
We refer the reader to \cite{BergLewisPule1986} for further discussion of Bose-Einstein condensation,
and let us now, in the same spirit, offer a precise definition of condensation in the context of random graphs.

Consider an increasing sequence of finite graphs $(\mc{G}_n)$, with vertex and edge sets $\mc{G}_n=(\mc{V}_n,\mc{E}_n)$,
in which each vertex $v$ has a fitness value $F_v\in[0,1]$.
We define the quantities
\begin{align}
\mu_n(A)&=\frac{1}{2|\mc{E}_n|}\sum\limits_{v\in \mc{V}_n} \deg_n(v)\1_{\{F_v\in A\}}, \label{eq:mu_graph} \\
\ell_n(A)&=\frac{1}{2|\mc{E}_n|}\max\limits_{v\in\mc{V}_n}\,\deg_n(v)\1_{\{F_v\in A\}}\label{eq:ell_def}.
\end{align}
Thus, $\mu_n$ is a random probability measure on $[0,1]$ which measures the fitnesses present in $\mc{G}_n$, weighted according to degree.
The quantity $\ell_n(A)$ is not a measure; it is the proportion of half-edges in $\mc{G}_n$ that are attached to the highest degree vertex with fitness in $A$.
\begin{defn}
\label{def:condensation}
Let $B_\epsilon(a)=[a-\epsilon,a+\epsilon]\cap[0,1]$.
\begin{enumerate}
\item We say that condensation occurs at $a$ if $\lim\limits_{\epsilon\to 0}\liminf\limits_{n\to\infty}\E\l[\mu_n(B_\epsilon(a))\r]>0$.
\item We say that condensation at $a$ is extensive if $\lim\limits_{\epsilon\to 0}\liminf\limits_{n\to\infty}\E\l[\ell_n(B_\epsilon(a)\r]>0$.
\item We say that condensation at $a$ occurs around the persistent hub $v$, if $v$ is a fixed vertex with fitness $a$ such that $\liminf\limits_{n\to\infty}\E\l[\frac{1}{|\mc{E}_n|}\deg_n(v)\r]>0$.
\end{enumerate}
\end{defn}
For many models with fitness, including PAC, the weak limit $\mu_n\to \mu$ exists almost surely and the limit $\mu$ is deterministic.
In such cases condensation at $a$ is equivalent to $\mu$ possessing an atom at $a$.
Extensive condensation occurs only when the degrees of individual vertices make non-negligible contributions to the formation of this atom.
These three definitions provide qualitative measures of how strongly the structure of $\mc{G}_n$ becomes dominated by a small fraction of high degree nodes, as $n\to\infty$. Clearly, 3 $\ra$ 2 $\ra$ 1.


As we have mentioned, we are interested in models for which condensation occurs either at $a=1$ or not at all.
In such cases, condensation occurs only through a positive fraction of the half-edges appearing on ever fitter vertices.
Extensive condensation captures the more specific event that, in the limit, individual vertices become (each perhaps only for a limited time) neighbouring to a positive fraction of the graph.

\begin{remark}
From now on, we use the term condensation to mean condensation at $1$.
\end{remark}

Let us now summarise the various techniques which have been used to rigorously analyse condensation in models of preferential attachment, with particular attention given to models incorporating fitness and/or choice.
Readers familiar with this literature may wish to move directly on to Section \ref{sec:popgen}, and will not miss out on any notation by doing so.

We first recall a natural coupling between the classic preferential attachment model and an urn process.
Fix, $v_0\in \mc{G}_0$. Colour $v_0$ white and all other vertices black; pass these colours on to the associated half-edges. Now, regard each half-edge of $\mc{G}_n$ as a coloured ball within an urn $\mc{U}_n$.
From the dynamics of classic preferential attachment, the one-step dynamics of $(\mc{U}_n)$ are as follows. To obtain $\mc{U}_{n+1}$ from $\mc{U}_n$, we:
\begin{enumerate}
\item Draw a ball uniformly at random from $\mc{U}_n$ and note its colour. Return this ball to the urn.
\item Add a new black ball to the urn, and also add a new ball of the same colour as was drawn in step 1.
\end{enumerate}
Then, at all times, the number of white balls in $\mc{U}_n$ is equal to $\deg_n(v_0)$.
The new black ball corresponds to the half-edge associated to a new vertex $v$; the drawn ball corresponds to sampling the (colour of the new half-edge attached to the) vertex to which $v$ connects.
It is straightforward to extend the coupling to track the joint degree of multiple balls, using multiple colours.

The first rigorous analysis of the Bianconi-Barab\'{a}si model was provided by \cite{BorgsEtAl2007}, 
who extended the idea described above to couple the model to a generalized P\'{o}lya urn process.
In a generalized P\'{o}lya urn each colour is assigned a different activity value (in this case, given by a function of the fitness).
Crucially, these activity values weight how balls are drawn from the urn, in a way that exactly matches the fitness-dependent sampling used in the Bianconi-Barab\'{a}si model.
With this coupling in hand, \citeauthor{BorgsEtAl2007} invoked the limit theory of urns provided by \cite{Janson2004}, and showed rigorously that condensation occurred. However, this limit theory applies only when the urn has finitely many colours, meaning that discretization of the fitness values was a necessary step within the proof.

As we have mentioned, \cite{BianconiBarabasi2001} predicted extensive condensation within their model.
This prediction was shown to be false by \cite{DereichEtAl2017}, who embedded the Bianconi-Barab\'{a}si model in continuous time (a technique advocated by \citeauthor{Janson2004}) and,
having done so, viewed it as a multi-type branching process with reinforcement.
In this formulation, half-edges correspond to individuals within the branching process, and having greater fitness corresponds to being a type of individual that branches at faster rate.
Individuals with the same fitness are referred to as a family.
(In fact \citeauthor{DereichEtAl2017} considered a more general case than \citeauthor{BianconiBarabasi2001}, by including an extra parameter controlling the rate at which new edges appear between existing vertices.)

The argument given by \citeauthor{DereichEtAl2017} for non-extensive condensation proceeds via computations based on the growth rates and birth times of families, utilising the independence inherent within branching processes.
Their result requires regular variation of the fitness distribution near $1$, which covers the range of parameters of interest to \citeauthor{BianconiBarabasi2001}.
For non-regularly varying fitness distributions the behaviour is not known, but see Section 8 of \cite{DereichEtAl2017} for a discussion.

The analysis of \cite{MalyshkinPaquette2014} relies heavily on the appearance of a persistent hub within their model.
It proceeds by first showing that the number of possible persistent hubs is almost surely finite, followed by showing that for any two vertices, which one has higher degree may switch only finite many times.
These arguments rely on comparisons to classical preferential attachment (which is also known to have a persistent hub).
With this information in hand, \citeauthor{MalyshkinPaquette2014} used stochastic approximation to analyse the growth of the persistent hub, which they show to have degree of asymptotic order $n$ when $R>2$ and order $\frac{n}{\log n}$ when $R=2$.

More generally, stochastic approximation is a well established method of studying urn processes and preferential attachment models. We refer the reader to the survey article of \cite{Pemantle2007} for details.
A rather general application of stochastic approximation to an extension of the Bianconi-Barab\'{a}si model can be found in \cite{DereichOrtgiese2014}.
We will discuss the applicability of stochastic approximation to PAC in Remark \ref{rem:stoch_approx}.

Some authors have considered variants of preferential attachment with choice in which the chosen vertex is not (or is not always) the fittest or the most valent of the $R$ samples.
Examples of such models, which have typically been studied through stochastic approximation, appear in \cite{MalyshkinPaquette2015} and \cite{HaslegraveEtAl2018}. The latter includes a particular example with $R=3$ and attachment to the vertex with middle fitness, in which condensation occurs at a \textit{random} location within $(0,1)$.

For models with choice, the coupling described earlier results in an urn process for which multiple balls must be drawn and reacted to on each time step.
\cite{Janson2004} comments that such urns are often intractable, however we will be able to analyse the urn process arising from PAC using the aforementioned duality.


\subsection{Multiple waves of natural selection}
\label{sec:popgen}

In populations genetics, models that feature multiple waves of natural selection towards ever fitter individuals, are rare.
To our knowledge, at the present time all known tractable examples are close relatives of the model introduced by \cite{DesaiFisher2007},
who described an extension of the Moran model in which mutations produce ever fitter individuals and selection brings the descendants of some of these individuals to dominance.
A detailed rigorous analysis, in the limit of large population size, was given recently by \cite{Schweinsberg2017};
see also the references therein for variants and special cases that were treated in earlier articles.

In loose terms, we may compare a wave of natural selection in which a fit sub-population emerges and grows to dominance (this is known as a selective sweep)
followed by their later demise in a subsequent even fitter wave,
to the growth and eventual decline of $\frac{1}{n}\deg_n(v)$, where $v$ is a fit vertex within PAC.
\cite{Schweinsberg2017} showed that within the Desai-Fisher model, and under suitable assumptions, 
the initial growth of each new wave could be approximated by a branching process.
However, this approximation breaks down once the new wave becomes a positive fraction of the total population, after which point a fluid limit is used.
The same paradigm can be found within the infectious disease literature, for example in \cite{BallSirl2017} (for a single wave of infection),
and also within the heuristics described for our own proofs in Section \ref{sec:proofs_outline}.
However, in our case \textit{time will be reversed} and we will be tracking the growth of the genealogies of half-edges.

There are substantial differences between the Desai-Fisher model and PAC.
In the Desai-Fisher model individuals die and are replaced, whereas
in PAC once a vertex has appeared it remains present forever.
Moreover, in the corresponding regime of the Desai-Fisher model, 
the individuals that cause the $(j+1)^{th}$ wave first appear during the $j^{th}$ wave, 
whereas within PAC we will see that a new fittest vertex born at time $\approx n^\beta$, where $\beta\in(0,1)$, will survive through several waves of dominance by less fit (but older) vertices, before it has its own chance at time $\approx n$.

\section{Results on Preferential Attachment with Choice by fitness}
\label{sec:our_model}

Let us now define the notation which, from now on, we use (only) for PAC. 
From hereon we refer to PAC as `the' model.
The model is parametrized by (the distributions of) a pair of random variables, $F$ taking values in $[0,1]$ and $R$ taking values in $\N$. 
For clarity, we use the convention that $\N=\{1,2,\ldots\}$, so $R$ does not take the value $0$.
Let $(F_n)$ be a sequence of i.i.d.~samples of $F$, and let $(R_n)$ be a sequence of i.i.d.~samples of $R$.

We describe an increasing sequence of random graphs $(\mc{G}_n)_{n\geq 0}$ with vertex and edge sets $\mc{G}_n=(\mc{V}_n,\mc{E}_n)$.
We begin from a graph $\mc{G}_0$, which we will take to be a single vertex $v_0$ with a self-loop.
In fact, our results hold for an arbitrary finite initial graph $\mc{G}_0$,
but we follow a common convention and make this choice for simplicity.

\begin{defn}
\label{def:graph_model}
The dynamics of $(G_n)$ are as follows.
At each time step we will add a single new vertex $v_n$ to the graph, so that $\mc{V}_n=\{v_0,v_1,\ldots,v_n\}$. The new vertex $v_n$ is assigned the fitness value $F_n$.
Given $\mc{G}_{n-1}$ and the fitnesses of its vertices, we attach $v_{n}$ according to the following rule.
\begin{enumerate}
\item First, we sample an ordered set of $R_{n}$ existing vertices, which we label as $p_{n,1},p_{n,2},\ldots,p_{n,R_{n}}$. We define the (unordered) set 
\begin{equation}\label{eq:parents}
P_n=\{p_{n,1},p_{n,2},\ldots,p_{n,R_{n}}\}.
\end{equation}

Each of the $p_{n,l}$ is sampled independently and with replacement, from $\mc{V}_n$, according to preferential attachment. That is, for each index $l=1,\ldots,R_n$, the probability of picking the vertex $v\in\mc{V}_n$ is proportional to $\deg_{n-1}(v)$.
\item A single new vertex $v_{n}$ joins the graph by attaching via a single new edge to the fittest vertex in $P_n$.
\end{enumerate}
We assume that the distribution of $F$ is absolutely continuous, with essential supremum $1$.
Consequently, distinct vertices have unique fitness values and step 2 is well defined.
\end{defn}

\begin{remark}
\label{rem:fitness_Gn}
Within $P_n$, which vertex is fittest depends on the order of the fitness values, but not on their specific values.
Thus, whilst $\mu_n$ defined by \eqref{eq:mu_graph} does depend on the distribution of $F$, in fact in PAC the distribution of the graph $\mc{G}_n$ does not.
\end{remark}
The key parameter in the model is the distribution of $R$, which affects both $\mu_n$ and $\mc{G}_n$. Heuristically, when $R$ tends to take larger values, we should expect that fit vertices will become more successful at capturing edges, thus making condensation more prone to occur.
We will assume, throughout, that
$$
\E[R]<\infty.
$$
We now state our results rigorously.
Our first result sets the scene, and shows that as $n\to\infty$ each vertex grows towards infinite degree but, whilst doing so, does not become a persistent hub.

\begin{theorem}\label{thm:fixed}
Let $v$ be a (deterministic) vertex. Then, $\deg_{\mc{G}_n}(v)\to\infty$ almost surely, and $\frac{1}{n}\deg_{\mc{G}_n}(v)\to 0$ in probability.
\end{theorem}
The next result describes the precise limiting distribution of the degree weighted fitnesses distribution $\mu_n$, as $n\to\infty$.
Of course, this results in a characterization of when condensation occurs.
The statement involves a particular Galton-Watson process which, as we have already mentioned, will play a key role in the proof.

\begin{theorem}\label{thm:condensation_graph}
Suppose that $\E[R^2]<\infty$.
Let $L$ be the number of leaves of a Galton-Watson tree, started with a single individual
and with offspring distribution $M$ given by
\begin{equation}\label{eq:gw_offspring}
\P[M=m]=
\begin{cases}
\frac12 & \text{ if }m=0\\
\frac12\P[R=m] & \text{ if }m\in\N.
\end{cases}
\end{equation}
Let $(C_n)_{n\in\N}$ be a sequence of i.i.d.~copies of $F$, independent of $L$.
Then, almost surely, as $n\to\infty$, $\mu_n$ converges weakly to the probability measure $\mu$ on $[0,1]$ given by
\begin{equation}\label{eq:limiting_colour}
\mu([0,a])=\frac12\P[F\leq a]+\frac12\l(\P\l[L<\infty\text{ and }\max_{i=1,\ldots,L}C_i \leq a\r]+\1_{\{a=1\}}\P[L=\infty]\r).
\end{equation}
where $a\in[0,1]$.
\end{theorem}

\begin{cor}\label{cor:condensation}
Suppose that $\E[R^2]<\infty$. Then condensation occurs if and only if $\E[R]>2$.
\end{cor}
Note that we do not observe condensation at $R=2$, when the Galton-Watson process is critical.
Criticality may well lead to other interesting behaviour, but we do not explore this possibility within the present article.
It is clear that Corollary \ref{cor:condensation} follows immediately from Theorem \ref{thm:condensation_graph}, because $\P[L=\infty]>0$ if and only if $\E[R]>2$.
We state our final result:
\begin{theorem}\label{thm:extensive_condensation_graph}
Suppose that $\E[R]>2$ and that $\E[R^2]<\infty$.
Then, extensive condensation occurs within the model.
\end{theorem}
Combining Theorems \ref{thm:fixed} and \ref{thm:extensive_condensation_graph}, we have that extensive condensation occurs without the formation of a persistent hub. Our proofs of the above theorems rely on a time-reversed duality, between $(\mc{G}_n)$ and the genealogy of an urn process $(\mc{U}_n)$, which is naturally coupled to $(\mc{G}_n)$ in the same style as described (for classical preferential attachment) in Section \ref{sec:background}.
The genealogy of $(\mc{U}_n)$ can in turn be coupled, but only for a limited time, to a Galton-Watson tree $\mc{T}^n$ with offspring distribution \eqref{eq:gw_offspring}.
We introduce these couplings in Sections \ref{sec:urn_coupling} and \ref{sec:gw_coupling}, 
to be followed by a heuristic outline of the proofs in Section \ref{sec:proofs_outline}.
The proofs themselves, of Theorems \ref{thm:fixed}, \ref{thm:condensation_graph} and \ref{thm:extensive_condensation_graph} are given in Sections \ref{sec:fixed_vertex_family}, \ref{sec:condensation} and \ref{sec:extensive_condensation} respectively.

In Section \ref{sec:general_addition} we discuss a natural extension to our results;
we consider the effects of incorporating a mechanism commonly used to control the strength of preference that incoming vertices have for making connections to high degree vertices.
In PAC this mechanism is closely related to attaching new vertices onto the existing graph via multiple new edges.


\subsection{Couplings and dualities}
\label{sec:couplings}

\subsubsection{Coupling to an urn process}
\label{sec:urn_coupling}

We define an urn process $(\mc{U}_n)$ which will be coupled to $(\mc{G}_n)$.
In the urn, each ball will have a colour, represented as a number in $[0,1]$, and this colour corresponds to a fitness value (of a vertex) in the graph model.
The balls themselves correspond to the half-edges of the graph. We write balls in bold case e.g.~$\v{u}$, and we write the colour of $\v{u}$ as $\col(\v{u})$.
From now on, \textit{we will use the terms fitness and colour interchangeably}.
Formally, let $\mc{U}_n$ be the set of half-edges in the graph $\mc{G}_n$, where $n\in\N_0$.
For each $\v{u}\in \mc{U}_n$, we set $\col(\v{u})$ to be the fitness of the vertex to which $\v{u}$ is attached.
\begin{defn}
\label{def:urn_model}
The dynamics of the process $(\mc{U}_n)$ are as follows.
Label the two initial half-edges in $\mc{G}_0$ as $\v{c}_0$ and $\v{s}_0$.
To construct $\mc{U}_{n}$, given $\mc{U}_{n-1}$, do the following:
\begin{enumerate}
\item Draw $R_n$ balls, independently and uniformly at random, from $\mc{U}_{n-1}$. Label these balls
\begin{equation}\label{eq:potential_parents}
\mc{P}_n=\l\{\v{p}_{n,1},\ldots,\v{p}_{n,R_n}\r\}.
\end{equation}
\item Let $\v{c}_n$ be a new ball with $\col(\v{c}_n)=\max\{\col(\v{p}_{n,l})\-l=1,\ldots,R_n\}$.\\
Let $\v{s}_n$ be a new ball with $\col(\v{s}_n)=F_n$.
\item Define $\mc{U}_{n}=\mc{U}_{n-1}\cup\{\v{c}_n,\v{s}_n\}$.
\end{enumerate}
\end{defn}
Using the notation above, we divide the balls within $\mc{U}_n$ into two distinct types: the \textit{cue} balls $\v{c}_n$ and \textit{source} balls $\v{s}_n$.
Recall that we take the initial graph $\mc{G}_0$ to be a single vertex with a self-loop. We extend the terminology of `cue' and `source' to $\mc{U}_0$, by writing $\mc{U}_0=\{\v{s}_0,\v{c}_0\}$ and specifying that $\v{s}_0$ is a source ball and $\v{c}_0$ is a cue ball.
We write $\mc{S}_n=\{\v{s}_0,\v{s}_1,\ldots,\v{s}_n\}$ and set $\mc{S}=\cup_n\mc{S}_n$.
We define $\mc{C}_n$ and $\mc{C}$ analogously for cue balls.
Thus, $\mc{U}_n=\mc{S}_n\cup\mc{C}_n$ and we set $\mc{U}=\mc{S}\cup\mc{C}$.

The process $\mc{U}_n$ is a projection of $\mc{G}_n$, in the sense that $\mc{U}_n$ forgets the graph structure and remembers only how many half-edges of each colour were present in $\mc{G}_n$. Nonetheless, $\mc{U}_n$ is a Markov process with respect to the filtration generated by $(R_n,F_n,\mc{P}_n)$.
Note that the random measure $\mu_n$ satisfies
\begin{equation}\label{eq:mu_urn}
\mu_n(A)=\frac{1}{|\mc{U}_n|}\sum\limits_{b\in\mc{U}_n}\1\{\col(b)\in A\}.
\end{equation}
Thus, $\mu_n(A)$ is the proportion of balls with colour $\in A$ at time $n$. 
We can therefore understand \eqref{eq:limiting_colour} as expressing the limiting distribution
of the colour of a ball drawn (uniformly) from the urn at large time.

\subsubsection{Representation as a genealogy}
\label{sec:genealogy_repr}

We equip the balls in the urn $\mc{U}$ with a genealogy that records the way in which each new cue ball $\v{c}_n$ inherits its colour from a single pre-existing ball.
We will use terminology from population genetics to describe this genealogy. The fitness values (i.e.~colours) play precisely the role of fitnesses in population models.

We say that $\mc{P}_n$ from \eqref{eq:potential_parents} are the \textit{potential parents} of $\v{c}_n$.
We refer to the unique ball in $\mc{P}_n$ with colour $\col(\v{c}_n)$ as the \textit{parent} of $\v{c}_n$.
We say that $\v{c}_n$ is a \textit{child} of its parent ball.
To handle time $n=0$, we say that $\v{s}_0$ is the parent of $\v{c}_0$, and we give $\v{c}_0$ precisely $R_0$ potential parents all of which are equal to $\v{s}_0$, where $R_0$ is an independent copy of $R$. Lastly, source balls do not have any parents or any potential parents.

A finite sequence $(\v{b}^{(k)})_{k=1}^K$ of balls in which, for all $k$, the ball $\v{b}^{(k+1)}$ is the parent of $\v{b}^{(k)}$ (resp.~potential parent of $\v{b}^{(k)}$), and in which $\v{b}^{(K)}$ is a source ball, is said to be the \textit{ancestral line} (resp.~a \textit{potential ancestral line}) of $\v{b}^{(1)}$.
We stress that each ball has a unique ancestral line, but multiple potential ancestral lines. 
Each potential ancestral line ends in a source ball, which necessarily has no potential parents.
Given any ball $\v{b}\in\mc{U}$, we write $\v{b}^\downarrow$ for the set of balls that appear in one or more of the potential ancestral lines of $\v{b}$, including $\v{b}$ itself.
The set $\v{b}^{\downarrow}$ is known as the set of \textit{potential ancestors} of $\v{b}$.
If we couldn't see the fitness values of the balls, but could see which balls made up the sets $\mc{P}_n$, then $\v{b}^{\downarrow}$ represents the full set of balls which might have been lucky enough to have their own fitness value passed on $\v{b}$. Thus
\begin{equation}\label{eq:duality}
\col(\v{b})=\max\l\{\col(\v{u})\-\v{u}\in \v{b}^\downarrow\r\}
=\max\l\{\col(\v{s})\-\v{s}\in \v{b}^\downarrow\cap\mc{S}\r\}.
\end{equation}
In words: the colour of $\v{b}$ is the colour of the fittest source ball within its potential ancestors.
This fact is a natural consequence of point 2 of Definition \ref{def:urn_model}.
See Figure \ref{fig:urn_gen} for a graphical explanation.

\begin{figure}[t]
\center{\includegraphics[scale=0.4]
{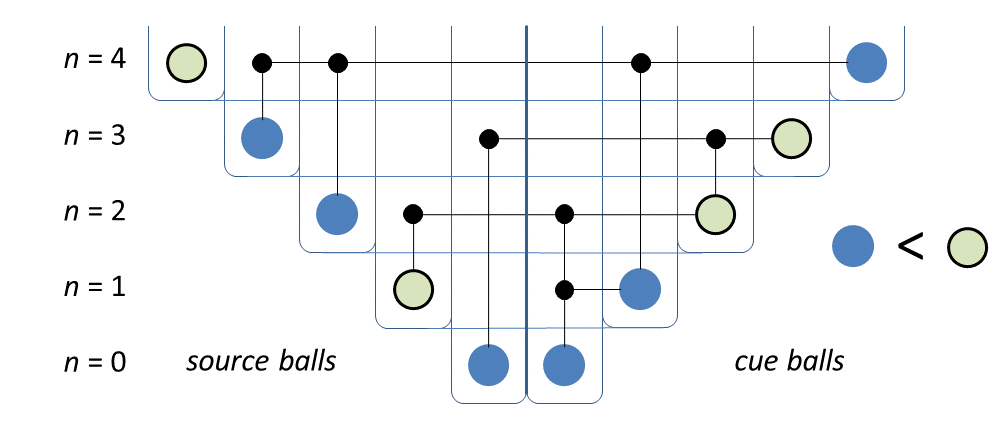}}
\caption{\label{fig:urn_gen} 
A graphical representation of the genealogy of the urn $(\mc{U}_n)$, in a case with only two different colours of balls. 
Columns correspond to balls, with source balls on the left and cue balls on the right.
Rows correspond to time-steps, with one new source ball and one new cue ball introduced on each row.
The relative fitness of the two colours is shown by an inequality. 
Smaller black dots correspond to potential parents chosen by cue balls, and the lines connecting them to balls represent the genealogy.
Looking backwards in time, branching is visible at each step of time when the new cue ball samples its potential parents, as black dots on the same row.
Coalescences occur when the same potential parent is sampled more than once (at possibly different times); one such event is visible
at $n=1,2$ within the column of the initial cue ball.
}
\end{figure}

Equation \eqref{eq:duality} is in the same spirit as the duality used (for a version of the spatial $\Lambda$-Fleming-Viot process) by \cite{EtheridgeFreemanEtAl2017}.
More generally, dualities of this kind are instances of ancestral selection graphs, introduced by \cite{KroneNeuhauser1997}. The selective mechanism of always choosing the potential parent with maximal fitness simplifies their structure considerably, whereas in general they can lead to quite intractable dual processes.

We will be particularly interested in the structure of $\v{c}^\downarrow_n$, when $n$ is large.
It is natural to view $\v{c}_n^\downarrow$ as a branching-coalescing structure: \textit{coalescence} of (potential) ancestral lines occurs when a given ball is a (potential) parent to more than one cue ball. Similarly, when a cue ball has more than one potential parent we say that it as a \textit{branching} of potential ancestral lines.

We write $\v{b}^{\uparrow}$ for the set of balls which contain $\v{b}$ within their ancestral line.
The set $\v{b}^\uparrow$ is known as the \textit{family} or \textit{descendants} of $\v{b}$.
When $\v{b}$ is a source ball we refer to $\v{b}$ as the \textit{founder} of the family $\v{b}^\uparrow$, and the members of this family at time $n$ are $\mc{U}_n\cap \v{b}^\uparrow$.
Note that all elements of $\v{b}^\uparrow$ have the same colour as $\v{b}$, and that if $v$ is the vertex to which the source ball $\v{s}_k$ is attached, then
\begin{equation}
\label{eq:urn_degrees}
\deg_{\mc{G}_n}(v)=|\mc{U}_n\cap\v{s}_k^\uparrow|.
\end{equation}
We stress that $\v{b}^\uparrow$ is based on ancestral lines, whereas $\v{b}^\downarrow$ is based on potential ancestral lines.
Hence, $\v{b}^\uparrow$ depends on the sequence of fitnesses $(F_n)$, but $\v{b}^\downarrow$ does not.

Theorem \ref{thm:fixed} states that $\P[|\v{b}^\uparrow|=\infty]=1$ and $\frac{1}{n}|\v{b}^\uparrow\cap\mc{U}_n|\to 0$ in probability, for any fixed ball $\v{b}\in\mc{U}$.
Theorem \ref{thm:extensive_condensation_graph} is more complex to translate, but note that it is implied if,
as $n\to\infty$, we see a non-vanishing probability that $\sup_{k\leq n} \frac{1}{n}|\v{s}_k^\uparrow\cap\mc{U}_n|$ has asymptotic order $n$. That is, the size of the largest family at time $n$ should be of order $n$.

\subsubsection{Coupling to a Galton-Watson process}
\label{sec:gw_coupling}

There is a natural coupling between the urn process $(\mc{U}_n)$ and a Galton-Watson process, which we will now describe.
This coupling is only valid for a limited time;
a Galton-Watson tree can only accurately represent the genealogy of $\v{c}_n$ as far backwards in time as that genealogy remains tree-like.
Let $\mc{W}^n_0=\{\v{c}_n\}$.
Then, iteratively, define
$$\mc{W}^n_{k}=\{\v{p}\-\v{p}\text{ was a potential parent of some }\v{b}\in\mc{W}^n_{k-1}\}.$$
Note that $\mc{W}^n_k$ is an (unordered) set, so that even if $\v{p}$ is a potential parent to more than one $\v{b}\in\mc{W}^{n}_{k-1}$, only one instance of $\v{p}$ appears in $\mc{W}^n_k$. Note also that $\mc{W}^n_k\sw\v{c}_n^\downarrow$ contains precisely the $k^{th}$ generation of potential ancestors of $\v{c}_n$
Since source balls have no potential ancestors it is possible for $\mc{W}^n_k$ to become empty.
Let
\begin{align*}
K_n&=\inf\l\{k\in\N\-\exists\,\v{p}_{i,l}=\v{p}_{i',l'}\text{ such that }(i,l)\neq (i',l')\text{ and }\v{c}_i,\v{c}_{i'}\in \mc{W}^n_0\cup\ldots\cup\mc{W}^n_{k-1}\r\}.
\end{align*}
In words, $K_n$ is the first generation of $\mc{W}^n_k$ in which we encounter a coalescence.
We use the usual convention that the empty set has $\inf\emptyset=\infty$, covering the case where no coalescences are encountered (which may occur when if $\mc{W}^n_k$ becomes empty).

\begin{lemma}\label{lem:gw_coupling}
Let $W^n_k=|\mc{W}^n_k|$, and let $(Z_k)$ be a Galton-Watson process with offspring distribution given by \eqref{eq:gw_offspring}.
Then there exists a coupling such that $(W^n_k)_{k=0}^{K_n-1}=(Z_k)_{k=0}^{K_n-1}$.
\end{lemma}
\begin{proof}
Let us first note that,  by Definition \ref{def:urn_model}, the potential ancestry of any given ball is independent of the fitness values of (all) balls. Thus fitnesses play no role in this proof; all quantities considered are independent of them.

Given the process $(W^n_k)$, we define
$Z_k=W^n_k$ for $k<K_n$, and for $k\geq K_n$ we allow $Z_k$ to evolve independently of $(W^n_k)$ as a Galton-Watson process with offspring distribution \eqref{eq:gw_offspring}.
It remains to show that $Z_k$ has the desired distribution for $j<K_n$.
To this end, let us consider when $k<K_n$ and $W^n_k=m$. Thus $Z_k=W^n_k$,
and we look to calculate the distribution of $Z_{k+1}$.
We must consider two cases.
\begin{itemize}
\item If $k+1=K_n$ then by definition of $Z_k$, then the transition $Z_k\mapsto Z_{k+1}$ will be independent of $(W^n_k)$ and will be that of a Galton-Watson process with offspring distribution \eqref{eq:gw_offspring}.
\item If $k+1<K_n$ then $Z_{k+1}=W^n_{k+1}$. We have $W^n_k=m$, so $\mc{W}^n_k$ contains $m$ (distinct) balls, but we do not know the identity of these balls. 
Moreover, because $k<K_n$, each such ball is not an element of $\mc{W}^n_0\cup\ldots\cup\mc{W}^n_{k-1}$.
Thus, each such ball is, independently of each other, and independently of $\mc{W}^n_0\cup\ldots\cup\mc{W}^n_{k}$, a cue ball with probability $\frac12$ and a source ball with probability $\frac12$.
By point 1 of Definition \ref{def:urn_model}, these cue balls each, independently, have i.i.d.~numbers of potential parents with common distribution matching that of $R$.
Because $k+1<K_n$, the identities of these potential parents are distinct.
Thus, $\mc{W}^n_{k+1}$ is the sum of $m$ i.i.d.~copies of $R$, and the transition $W^n_k\mapsto W^n_{k+1}$ is that of a Galton-Watson process with offspring distribution \eqref{eq:gw_offspring}.
\end{itemize}
In both cases, the transition $Z_k\mapsto Z_{k+1}$ has the desired distribution. 
\end{proof}


We will show in Lemma \ref{lem:Kn_control} that $\P[K_n\leq k]\to 0$ as $n\to\infty$, for all $k\in\N$. In words, as $n\to\infty$ the coupling of $\v{c}^\downarrow_n$ to a Galton-Watson tree remains valid for an arbitrarily large $\mc{O}(1)$ number of generations of this tree.

\subsection{Outline of proofs}
\label{sec:proofs_outline}

All of our proofs rely on the couplings detailed above. The proof of Theorem \ref{thm:extensive_condensation_graph} relies on analysing the genealogy of $(\mc{U}_n)$ directly, whereas Theorem \ref{thm:condensation_graph} uses only the Galton-Watson coupling, and Theorem \ref{thm:fixed} uses both. We outline all three proofs in this section.

\medskip

Let us discuss Theorem \ref{thm:fixed} first.
In terms of the urn process, the first part of Theorem \ref{thm:fixed} asserts that $\P[|\v{u}^\uparrow|=\infty]=1$.
The proof rests on the observation that, when $\mc{P}_n$ is sampled, then for any fixed ball $\v{u}$, the probability of $\v{u}\in\mc{P}_n$ is of order $\frac{1}{n}$ as $n\to\infty$.
If we could apply the Borel-Cantelli lemma then, with a little extra work we could deduce that (almost surely) $\v{u}$ was a parent infinitely often, thus $|\v{u}^\uparrow|=\infty$.
Unfortunately, the lack of independence means the Borel-Cantelli lemma does not apply; instead we will use the Kochen-Stone lemma.

The second part of Theorem \ref{thm:fixed} asserts that $|\v{u}^\uparrow\cap\mc{U}_n|/|\mc{U}_n|\to 0$ in probability.
Because this quantity has expectation $(\1_{(\v{u}\in\mc{S})}+\sum_0^n \P[\v{c}_j\in \v{u}^\uparrow])/(2n+2)$, the zero limit is implied if $\P[\v{c}_n\in\v{u}^\uparrow]\to 0$.
To prove the latter, we use that the genealogy of $\v{c}^\downarrow_n$ is that of a Galton-Watson tree, at least for a large $\mc{O}(1)$ number of generations. If this Galton-Watson tree dies out (i.e.~in $\mc{O}(1)$ generations) then it has bounded size and is unlikely to include any fixed ball, in particular $\v{u}$. If it does not die out, then $\v{c}^\downarrow_n$ will include many source vertices, at least one of which is likely to be fitter than $\v{u}$. In both cases, $\v{c}_n\notin \v{u}^\uparrow$.

\medskip

Theorem \ref{thm:condensation_graph} establishes the limiting distribution of colours present in $\mc{U}_n$. Our proof first establishes the result in the case where only a two element set $\{0,1\}$ of colours are permitted. It is straightforward to upgrade this case into Theorem \ref{thm:condensation_graph}.
The argument for the two colour case relies on establishing the distribution of $\col(\v{c}_n)$ as $n\to\infty$.
Heuristically, as $n\to\infty$, we again compare $\v{c}^\downarrow_n$ to a Galton-Watson tree, and again the extinction/non-extinction dichotomy is key.
If the Galton-Watson tree dies out, then the colour of $\v{c}_n$ is the maximal colour of the source balls at its leaves.
If it does not die out, then $\v{c}^\downarrow_n$ contains many generations, which will mean that high probability there will be a source ball of maximum colour (i.e.~colour $1$, in the two colour case) within $\v{c}^\downarrow_n$, in which case $\col(\v{c}_n)=1$.
Recalling that half of all balls are cue balls, and the other half sources, along with \eqref{eq:gw_offspring} these considerations lead directly to the formula \eqref{eq:limiting_colour} given in Theorem \ref{thm:condensation_graph}.
The first term on the right of \eqref{eq:limiting_colour} represents the i.i.d.~colours of source balls, 
the latter term represents the cue balls.
\medskip

The proof of Theorem \ref{thm:extensive_condensation_graph}, given in Section \ref{sec:extensive_condensation}, takes up the majority of the present article. The outline is as follows.
Finding the largest family at time $n$ is essentially the same as identifying which source $\v{s}_k$, for $k\leq n$, was most likely to have founded the family to which $\v{c}_n$ belongs. This, in turn, relies on understanding the behaviour of the genealogy of $\v{c}_n^\downarrow$ during the stage at which it stops being tree-like, and coalescences start to have a significant effect.
Thus, we are trying to examine a property of the model that the Galton-Watson coupling will not capture.
For this reason, the Galton-Watson coupling is not used in the proof of Theorem \ref{thm:extensive_condensation_graph}, and we do not think in terms of the associated parent-child generations.

Consider the urn process $(\mc{U}_k)_{k=0}^n$. Looking backwards in time, as $k$ decreases from $n$ to $0$, we will see that at around time $k\approx n^\beta$, where
$$\beta=\frac{\E[R]-2}{\E[R]},$$
there starts to be a positive probability that a potential parent sampled for $\v{c}_k$ will already have been sampled as a potential parent of some $\v{c}_j\in \v{c}^\downarrow_n$ for $k<j\leq n$.
More precisely, we will follow the process 
$$\mc{H}^n_k=\{\v{u}\in\mc{U}_{k}\-\v{u}\in\mc{P}_j\text{ for some }\v{c}_j\in\v{c}^\downarrow_n\text{ with }j\geq k+1\}.$$
backwards in time (i.e~for $k=n,n-1,\ldots,0$).
Note that $\mc{H}^n_\cdot$ may both grow and shrink in size. During the transition $k+1\mapsto k$, the set $\mc{H}^n_\cdot$ will lose $\v{s}_k$ and $\v{c}_k$ if they were present, but if $\v{c}_k$ was present then any potential parents of $\v{c}_k$ that were not already present will be added in.

We denote the number of elements in the set $\mc{H}^n_k$ by $|\mc{H}^n_k|$.
We will see that when $k\approx n^\beta$,  $|\mc{H}^n_k|$ is of order $k$. 
Consequently, at this point a potential parent of $\v{c}_{k+1}$ has a non-negligible chance of being in $\mc{H}^n_k$;
thus coalescence becomes non-negligible.
In fact, the force of coalescence very quickly becomes strong, with the consequence that for $k\ll n^\beta$ essentially the entire urn $\mc{U}_k$ will be included in $\v{c}_n^\downarrow$, and in particular essentially all sources $\v{s}_j$ with $j\ll n^\beta$ will be included.
However, the fittest source ball $\v{s}_j\in\v{c}^\downarrow_n$ will, with high probability, be born during a `critical window' of time, $j\in[cn^\beta, Cn^\beta]$ where $c>0$ is small and $C<\infty$ is large.

We now summarise the techniques within the proof. When $k\gg n^\beta$ we will use iterative arguments, backwards in time, to construct bounds on the $\N\cup\{0\}$ valued process $|\mc{H}^n_k|$.
The resulting bounds on $|\mc{H}^n_k|$ will eventually break down, because in order to stay tractable they will partially ignore coalescences. However, they will stretch just far enough to see that, when $k\approx Cn^\beta$ for suitably large $C$, the set $\mc{H}^n_k$ comprises a small but non-negligible fraction of $\mc{U}_k$, with positive probability.

We then switch techniques, and for $k\in[cn^\beta, Cn^\beta]$ we aim to establish a fluid limit for the $[0,1]$ valued process $k\mapsto|\mc{H}^n_k|/|\mc{U}_k|$, in reverse time as $k$ decreases. After a suitable rescaling of time, this limit turns out to be an ordinary differential equation, with a stable fixed point at $1$ and an unstable fixed point at $0$; so starting just above zero results in attraction towards $1$. Having established the ODE limit, the key question becomes whether the critical window is actually long enough for the ODE to escape from $0$. Using an artificially longer critical window, by e.g.~taking a larger value of $C$, does not help because this results in an initial condition closer to zero.
However, on escaping $0$, we obtain non-vanishing behaviour for $|\mc{H}^n_k|/|\mc{U}_k|$ during $k\in[cn^\beta, Cn^\beta]$, which results in a positive probability that $\v{s}_k\in\v{c}_n^\downarrow$.

The final step of the proof involves combining the above results with the records process of $\col(\v{s}_k)$. We show that an unusually fit source ball born at around time $k\approx n^\beta$ may start a family that grows to include a non-vanishing proportion of $\mc{U}_n$, as $n\to\infty$.
Thus, extensive condensation occurs.


\begin{remark}\label{rem:stoch_approx}
Let us briefly survey what we might achieve via alternative methods.
For PAC, the techniques used by \cite{MalyshkinPaquette2014} are unavailable because a persistent hub does not emerge.
The techniques used by \cite{DereichEtAl2017} to analyse the Bianconi-Barab\'{a}si model are not available either, because we do not have independently growing families.

It is possible to use stochastic approximation to recover Theorem \ref{thm:condensation_graph}, but doing so results in a description of $\mu$ through the fixed points of a family of differential equations.
This is much less appealing than the intuitive formula \eqref{eq:limiting_colour} provided by the Galton-Watson coupling.

By contrast, it does not seem feasible to prove Theorem \ref{thm:extensive_condensation_graph} via stochastic approximation.
The vertex with greatest degree switches identity infinitely often and this greatly increases the amount of information which must be tracked.
Our attempts to find an alternative proof along such lines resulted in requiring more detailed information about the sensitivity of rather general families of ODEs to small perturbations than we were able to extract.
We discuss this issue a little further, after the key proof, in Remark \ref{rem:ext_vs_non}
\end{remark}


\section{Proof of Theorem \ref{thm:fixed}}
\label{sec:fixed_vertex_family}

In this section we prove Theorem \ref{thm:fixed} which, re-phrased in terms of the urn process $(\mc{U}_n)$, is split across two lemmas:
we prove that $\P[|\v{u}^\uparrow|=\infty]=1$ in Lemma \ref{lem:family_inf}, and that $|\v{u}^{\uparrow}\cap\mc{U}_n|/|\mc{U}_n|\to 0$ in probability in Lemma \ref{lem:family_zero_propn}.

\begin{lemma}\label{lem:family_inf}
Let $\v{u}\in\mc{U}$. Then $\P[|\v{u}^\uparrow|=\infty]=1$.
\end{lemma}

\begin{proof}
We consider the case of $\v{u}=\v{s}_0$ and suppose that $\col(\v{s}_0)=\alpha >0$.
It is easily seen that the argument for this case can be adapted to a general ball $\v{u}$.
Let
$$A_n=\{\v{p}_{n,1},\ldots,\v{p}_{n,R_n}\text{ are all source balls}\}\cap\{\v{p}_{n,1}\text{ is the fittest of the }\v{p}_{n,j}\}\cap\{\v{p}_{n,1}=\v{s}_0\}.$$
Note that, for any $n$, the probability that a (given) potential parent is both a source ball and less fit than $\v{s}_0$ is $\frac{\alpha}{2}$. Note also that $\P[\v{p}_{n,1}=\v{s}_0]=\frac{1}{2(n+1)}$, from which it is easily seen that $\P[A_n]$ has order $\frac{1}{n}$.

We will prove the present lemma by showing that $A_n$ occurs infinitely often.
Since the $A_n$ are correlated we will use a version of the Kochen-Stone lemma:
if $(E_n)$ are events such that $\sum_{n=1}^\infty \P[E_n]=\infty$ then
\begin{equation}\label{eq:kochen-stone}
\limsup_{N\to\infty}\frac{\sum_{1\leq m<n<N}\P[E_n]\P[E_m]}{\sum_{1\leq n<m\leq N}\P[E_n\cap E_m]}\leq \P[E_n\text{ infinitely often}].
\end{equation}
This result can be found as Theorem 1 of \cite{Yan2006}.
We will take $E_n=A_{i_n}$, where $i_n$ is defined as follows. Let $r=\inf\{r\in\N\-\P[R=r]>0\}$ and set $q=\P[R=r]$. Define
$i_0=0$ and
$$i_{n+1}=\inf\{l\in\N\-l>i_n,\;R_l=r,\;\text{and the }(\v{p}^l_j)_{j=1}^r\text{ are distinct source balls}\}.$$
The events $\{R_n=r\text{ and }(\v{p}_{n,j})_{j=1}^r\text{ are distinct}\}$ are mutually independent for different values of $n$. Moreover, for any $\epsilon>0$, for large enough $n$ the chance of the $(\v{p}_{n,j})_{j=1}^r$ being distinct is at least $1-\epsilon$, and the chance of them being distinct source balls is at least $(\frac{1}{2})^r-\epsilon$. Therefore, it follows from the strong law of large numbers that
$((\tfrac12)^r-\epsilon)q\leq \liminf_n\frac{i_n}{n}\leq \limsup_n\frac{i_n}{n}\leq (\tfrac12)^rq$ a.s.~and thus, since $\epsilon>0$ was arbitrary,
\begin{align}\label{eq:in_conv}
\P\l[\tfrac{i_n}{n}\to q2^{-r}\r]=1.
\end{align}
Until further notice, we condition on the sequence $(i_n)$ and work with the conditional measure $\P'[\cdot]=\P[\cdot\|\sigma(i_0,i_1,\ldots,)]$. Note that, under $\P'$, the $(\v{p}_{i_n,j})_{j=1}^r$ are conditioned to be distinct source balls, and thus are distributed as a uniformly random subset of $\{\v{s}_0,\ldots,\v{s}_{i_{n-1}}\}$ of size $r$.

We have $\P'[A_{i_n}]=\frac{1}{i_n}\alpha^{r-1}$. Here, the term $\frac{1}{i_n}$ is the probability of $\v{p}_{i_n,1}=\v{s}_0$ (given that $\v{p}_{i_n,1}$ is a source ball) and $\alpha^{r-1}$ is the probability that the other potential parents are all with fitness less than $\alpha$ (given that they are distinct source balls).
We now consider $\P'[A_{i_m}\cap A_{i_n}]$, where $m\leq n$. The probability that $\v{p}_{i_n,1}=\v{p}_{i_m,1}=\v{s}_0$ is $\frac{1}{i_n}\frac{1}{i_m}$. Hence, given $\v{p}_1^{i_n}=\v{p}_1^{i_m}=\v{s}_0$, the probability that $\v{c}_{i_n}$ and $\v{c}_{i_m}$ have no other common potential parents is $\binom{i_n-r-1}{r-1}/\binom{i_n-1}{r-1}$; a short elementary calculation shows that this probability is bounded between $1-\frac{2r^2}{i_n}$ and $1$. Thus,
\begin{align*}
\P'[A_{i_m}\cap A_{i_n}]
&=\frac{1}{i_n}\frac{1}{i_m}\alpha^{r-1}\alpha^{r-1}\l(1+\mc{O}\l(\tfrac{1}{i_n}\r)\r).
\end{align*}
Here, as usual, $f_n=\mc{O}(g_n)$ means that $\limsup_n|f_n/g_n|<\infty$.
Putting these together and cancelling factors of $\alpha$, in view of \eqref{eq:kochen-stone} we are interested to calculate the limit as $N\to\infty$ of
\begin{align*}
I_N=
\frac
{\sum\limits_{1\leq m<n\leq N}\frac{1}{i_n}\frac{1}{i_m}}
{\sum\limits_{1\leq m<n\leq N}\frac{1}{i_n}\frac{1}{i_m}\l(1+\mc{O}\l(\tfrac{1}{i_n}\r)\r)}
\end{align*}
By \eqref{eq:in_conv}, for all $\epsilon>0$, there exists (deterministic) $\mc{N}\in\N$ such that, with probability at least $1-\epsilon$, for all $n\geq \mc{N}$ we have $(1-\epsilon)q2^{-r}\leq \frac{i_n}{n}\leq q2^{-r}(1+\epsilon)$. On this event we have
$$
\frac{(1-\epsilon)^2}{(1+\epsilon)^2}
\liminf_{N\to\infty}
J_N
\leq \limsup_{N\to\infty}I_N \leq
\frac{(1+\epsilon)^2}{(1-\epsilon)^2}
\limsup_{N\to\infty}
J_N
$$
where
$$J_N=\frac
{\sum\limits_{1\leq m<n\leq N}\frac{1}{n}\frac{1}{m}}
{\sum\limits_{1\leq m<n\leq N}\frac{1}{n}\frac{1}{m}\l(1+\mc{O}\l(\tfrac{1}{n}\r)\r)}.
$$
It is easily seen that $J_N\to 1$ as $N\to\infty$, and since $\epsilon>0$ was arbitrary we conclude that also $I_N\to 1$. We thus have \eqref{eq:kochen-stone} (with $E_n=A_{i_n}$), and hence $\P'[A_{i_n}\text{ infinitely often}]=1$. Hence also $\P[A_{n}\text{ infinitely often}]=1$.
\end{proof}

We write $\mc{T}^n_k=\cup_{k=0}^n\mc{W}^n_k$ and $\mc{T}^n=\cup_{k=0}^\infty\mc{T}^n_k$.
Note that $\mc{T}^n=\v{c}_n^\downarrow$, which we accept as a small piece of redundancy in our notation.
We write $\mc{L}^n_k=\mc{T}^n_k\cap\mc{S}$ for the set of source balls in $\mc{T}^n_k$. Note that this is similar too, but not quite the same as, the set of leaves of $\mc{T}^n_k$; because $\mc{T}^n_k$ is curtailed at generation $k$, it may also have a number of cue-balls amongst its $k^{th}$ generation leaves. However, all leaves of $\mc{T}^n$ are source balls.

\begin{lemma}\label{lem:Kn_control}
Let $k\in\N$.
For all $\epsilon>0$ there exists $\delta>0$ and $N\in\N$ such that for all $n\geq N$,
$$\P\l[K_n> k\text{ and }\mc{T}^n_k\cap\mc{U}_{\lfloor\delta n\rfloor}=\emptyset\r]\geq1-\epsilon.$$
In particular, $\P\l[K_n> k\r]\to 1$ as $n\to\infty$.
\end{lemma}
\begin{proof}
We remark that if no coalescences occurred in $\v{c}^\downarrow_n$, then $K_n=\infty$ and the statement of the lemma holds trivially.
Let us refer to the single element of $\mc{W}^n_0$ as the `root'.
Fix $k\in\N$.
Since $\P[R<\infty]=1$,
it is easily seen that by choosing suitably large $A\in\N$ we obtain $\sup_n\P[|\mc{T}^n_{k}|\geq A]\leq\epsilon$.
For each $\v{b}\in\mc{T}^n_j$ there is a potential ancestral line, containing at most $k+1$ balls, between $\v{b}$ and the root.
Following this ancestral line backwards in time, the potential parents were chosen uniformly at random from the current urn.
By choosing $\delta>0$ small, we may control the chance that any of the (at most $k+1$) such potential parents along this line were sampled from within $\mc{U}_{\lfloor\delta n\rfloor}$.
Thus, we may choose $\delta\in(0,1)$ and $N\in\N$ such that $\sup_{n\geq N}\P[\mc{T}^n_k\cap\mc{U}_{\lfloor\delta n\rfloor}\neq\emptyset]\leq\epsilon$.  

Conditional on the event $\{\mc{T}^n_k\cap\mc{U}_{\lfloor\delta n\rfloor}=\emptyset\text{ and }|\mc{T}^n_{k}|\leq A\}$, each potential parent of each element of $\mc{T}^n_k$ was sampled uniformly from a set of balls with at least $\delta n$ elements.
The expected number of such potential parents is $\mc{O}(A\E[R])=\mc{O}(1)$, and the chance of choosing any particular ball as a potential parent is $\mc{O}(\frac{1}{\delta n})$.
Hence, the probability of seeing the same parent twice tends to zero as $n\to\infty$, and consequently $\P[K_n\leq k]\to 0$ as $n\to\infty$.
The result follows.
\end{proof}

\begin{lemma}\label{lem:leaf_production}
For all $k,n\in\N$, it holds that
$\P\l[|\mc{L}^n_k|< k/2\text{ and }\mc{W}^n_k\neq\emptyset\r]\leq (\tfrac12)^{k/2}$
\end{lemma}
\begin{proof}
Let $A^n_k$ denote the event that there is a potential ancestral line of $\v{c}_n$ containing at least $k$ cue balls, and that at least $k/2$ of these cue balls had no source balls amongst their potential parents.
If $\mc{W}^n_k$ is non-empty then, by definition of $\mc{W}^n_k$, there must be a potential ancestral line of $\v{c}_n$ that intersects $\mc{W}^n_k$.
Note that this potential ancestral line contains $k$ cue balls, corresponding to $k$ generations of $\v{c}^{\downarrow}_n$.
If, additionally, $|\mc{L}^n_k|< k/2$ then the event $A^n_k$ must occur.
In summary, $\{|\mc{L}^n_k|< k/2\text{ and }\mc{W}^n_k\neq\emptyset\}\sw A^n_k$.

Each potential parent has probability $\frac12$ of being a source ball. 
Hence, for all $j$, $\P[\mc{P}_j\cap\{\v{s}_0,\v{s}_1,\ldots\}=\emptyset]\leq 1/2$.
Since a potential ancestral line cannot include the same cue ball twice, $\P[A^n_k]\leq (1/2)^{k/2}$.
The stated result follows.
\end{proof}

\begin{lemma}\label{lem:family_zero_propn}
Let $\v{u}\in\mc{U}$. Then $\frac{|\v{u}^{\uparrow}\cap\mc{U}_n|}{|\mc{U}_n|}\to 0$ in probability as $n\to\infty$.
\end{lemma}
\begin{proof}
We will show $L^1$ convergence to zero, which is equivalent to convergence in probability because $|\v{u}^{\uparrow}\cap\mc{U}_n|/|\mc{U}_n|\leq 1$. Note that
$$\E\l[\frac{|\v{u}^\uparrow\cap\mc{U}_n|}{|\mc{U}_n|}\r]=\frac{1}{|\mc{U}_n|}\l(\1_{\l(\v{u}\in\mc{S}\r)}+\sum\limits_{k=0}^n\P[\v{c}_k\in \v{u}^\uparrow]\r).$$
Since $|\mc{U}_n|=2n+2$, it suffices to show that $\P[\v{c}_n\in \v{u}^{\uparrow}]\to 0$ as $n\to\infty$.

Note that $\P[F_1< F_2]$ is the probability that one source ball has fitness strictly less than that of another.
Since the fitnesses are independent, we have $\P[F_1< F_2]\leq \frac{1}{2}$.
Thus
\begin{equation}\label{eq:last_long_or_lose_2}
\P\l[\v{c}_n\in\v{u}^\uparrow\text{ and }|\mc{L}^n_k|\geq k/2\r]\leq (\tfrac{1}{2})^{k/2}
\end{equation}
because, on the event  that $\v{c}_n\in\v{u}^\uparrow$ and $|\mc{L}^n_k|\geq k/2$, at least $k/2$ source balls in $\mc{T}^n_k$ must either have fitness strictly less than that of $\v{u}$.
Let $\epsilon>0$. Let $\delta>0$, $k\in\N$, $N\in\N$, to be chosen shortly (dependent upon $\epsilon$).
For all $n\geq N$ we have
\begin{align}
\P\l[\v{c}_n\in\v{u}^\uparrow\r]
&\leq \epsilon+\P\l[\v{c}_n\in\v{u}^\uparrow,\; K_n\geq k\text{ and }\mc{T}^n_k\cap\mc{U}_{\lfloor\delta n\rfloor}=\emptyset\r]\notag\\
&\leq \epsilon+\P\l[\v{c}_n\in\v{u}^\uparrow,\; K_n\geq k\text{ and }\mc{W}^n_k\neq\emptyset\r].\notag\\
&\leq \epsilon+\l(\tfrac{1}{2}\r)^{k/2}+\P\l[\v{c}_n\in\v{u}^\uparrow,\; K_n\geq k\text{ and }\mc{W}^n_k\neq\emptyset\text{ and }|\mc{L}^n_k|\geq k/2\r]\notag\\
&\leq \epsilon+\l(\tfrac{1}{2}\r)^{k/2}+\l(\tfrac{1}{2}\r)^{k/2}\notag
\end{align}
The first line of the above follows from Lemma \ref{lem:Kn_control}, from which we obtain $N$ and $\delta$.
By increasing $N$, if necessary, we may assume that $\v{u}\in\mc{U}_{\lfloor\delta n\rfloor}$.
The second line then follows because, given that $\v{c}_n\in\v{u}^\uparrow$ and $\mc{T}^n_k\cap\mc{U}_{\lfloor\delta n\rfloor}=\emptyset$, the ancestral line linking $\v{c}_n$ to $\v{u}$ must extend beyond $\mc{T}^n_k$, and in particular $\mc{W}^n_k$ must be non-empty.
The third line then follows by Lemma \ref{lem:leaf_production}.
The final line follows from \eqref{eq:last_long_or_lose_2}. Choosing $k$ large enough that $2(\frac12)^{k/2}<\epsilon$ obtains that for all $n\geq N$, $\P\l[\v{c}_n\in\v{u}^\uparrow\r]\leq 3\epsilon$. This completes the proof.
\end{proof}

\section{Proof of Theorem \ref{thm:condensation_graph}}
\label{sec:condensation}

In this section we prove Theorem \ref{thm:condensation_graph}.
Throughout Section \ref{sec:condensation} we will adopt the conditions and notation used in the statement of Theorem \ref{thm:condensation_graph} 
In particular, let $L$ be the number of leaves on a Galton-Watson tree with offspring distribution \eqref{eq:gw_offspring}
and let $\mu$ be the measure on $[0,1]$ defined by \eqref{eq:limiting_colour}.
Let $(C_i)$ be a sequence of i.i.d.~copies of $F$.

Our proof proceeds by first establishing Theorem \ref{thm:condensation_graph} for a fitness distribution $F$ with only two possible values, $0$ and $1$.
Note that, as defined in Section \ref{sec:our_model}, the model does not currently allow for such a case because we had specified that the fitness distribution $F$ must be continuous on $[0,1]$. 
For general $F$, the extra difficulty is that we must handle the possibility that there may not be a unique fittest vertex (resp.~ball) within $P_n$ (resp.~$\mc{P}_n$), defined by \eqref{eq:parents} (resp.~\eqref{eq:potential_parents}).
This extra difficulty is not more than an irritation, which is why we excluded it in Section \ref{sec:our_model}.
It is convenient to first describe the case of non-absolutely continuous $F$ at the level of the urn process $\mc{U}_n$
-- which, we recall, is a projection of the graph $\mc{G}_n$ that records degrees via \eqref{eq:urn_degrees} but forgets the rest of the graph structure.
We then show how to reconstruct $(\mc{G}_n)$.

To define $\mc{U}_n$, 
Definition \ref{def:urn_model} still applies exactly as written,
but to define the associated genealogy we must specify how parent-child relationships are defined in the (additional) case that the potential parents $\mc{P}_n$ do not contain a unique fittest ball.
If there is not a unique fittest element of $\mc{P}_n$, 
then the parent of $\v{c}_n$ is chosen uniformly at random from the fittest balls within $\mc{P}_n$. 
Subject to this extra rule, the genealogical structure in Section \ref{sec:genealogy_repr} remains well defined.
Let us denote the parent of $\v{c}_n$ as simply $\v{q}_n\in\mc{P}_n$.

To define $\mc{G}_n$, as before we will take the balls of $\mc{U}_n$ to be set of half-edges of $\mc{G}_n$. 
We will specify how to form these half-edges into a graph, conditionally given the various processes $(\mc{U}_k, (\v{p}_{k,l}), \v{q}_n, R_k, F_k)_{k=0}^\infty$.
We will proceed inductively.
As before, we take $\mc{G}_0$ to be a single vertex with a self-loop, and $\mc{U}_0=\{\v{c}_0,\v{s}_0\}$ contains two balls of the same colour corresponding to two half edges of the same vertex.
Given $\mc{G}_{n-1}$, we already know the vertex set $\mc{V}_{n-1}$ and we know which half-edges within $\mc{U}_{n-1}$ are attached to which vertices. We attach a single new vertex $v_n$ via single new edge, as follows. 
\begin{enumerate}
\item For $l=1,\ldots,R_n$ we define $p_{n,l}\in \mc{V}_{n-1}$ to be the vertex attached to the half-edge $\v{p}_{n,l}\in\mc{U}_{n-1}$.
We define $P_n=\{p_{n,1},\ldots,p_{n,R_n}\}$.
Thus the $(p_{n,l})_{l=1}^{R_n}$ are i.i.d.~degree-weighted samples from $\mc{V}_{n-1}$.

\item We attach the half-edge $\v{c}_n$ to same vertex as its parent half-edge $\v{q}_n\in\mc{U}_{n-1}$ is already attached to (within $\mc{G}_{n-1}$). We specify that $(\v{c}_n,\v{s}_n)$ will together comprise a new edge, and we attach $\v{s}_n$ to a new vertex $v_n$. This new vertex is assigned fitness $F_n$.
\end{enumerate}
It is immediate that, when $F$ is continuous, the above mechanism precisely matches Definition \ref{def:graph_model}.
Moreover, it preserves the connections \eqref{eq:mu_urn} and \eqref{eq:urn_degrees} between $\mc{G}_n$ and $\mc{U}_n$.
In Section \ref{sec:cond_01} we will apply Lemmas \ref{lem:gw_coupling}, \ref{lem:Kn_control} and \ref{lem:leaf_production} in this extended context. Their proofs go through exactly as before -- in fact this is immediate because they were concerned only
with potential ancestors, the identities of which are unaffected by fitness values.
Restricting to only two colours, the equivalent statement to Theorem \ref{thm:condensation_graph} is as follows.

\begin{prop}\label{prop:condensation_urn_finite}
Take the fitness space $\mathscr{F}=\{0,1\}$, where $\P[F=0]$ and $\P[F=1]$ are both positive. Then for $a=0,1$,
\begin{equation}\label{eq:condensation_urn_finite}
\mu_n([0,a])\stackrel{a.s.}{\to}\frac12\P[F\leq a]+\frac12\l(\P\l[L<\infty\text{ and }\max_{i=1,\ldots,L}C_i\leq a\r]+\1_{\{a=1\}}\P[L=\infty]\r).
\end{equation}
\end{prop}
With Proposition \ref{prop:condensation_urn_finite} in hand, it is straightforward to deduce Theorem \ref{thm:condensation_graph}. We give this argument first, to be followed by the proof of Proposition \ref{prop:condensation_urn_finite}.

\begin{proof}[Of Theorem \ref{thm:condensation_graph}, subject to Proposition \ref{prop:condensation_urn_finite}.]
Recall that Theorem \ref{thm:condensation_graph} assumes a uniform fitness distribution on $[0,1]$.
Fix $a\in[0,1)$. Define $f(x)=\1\{x>a\}$, and define a new, two colour, urn process $\wt{\mc{U}}_n$, with the same set of balls as $\mc{U}_n$ and the same distribution for $R_n$, by considering balls with fitness $x$ to have the new fitness $\wt{x}=f(x)$. Thus, our new urn process has fitness space $\{0,1\}$ and fitness distribution $\wt{F}$ satisfying
$\P[\wt{F}=0]=\P[F\in[0,a]]$, $\P[\wt{F}=1]=\P[F\in(a,1]]$.
Let us write $\wt\mu_n$ for the empirical measure of colours within $\wt{\mc{U}}_n$, analogous to \eqref{eq:mu_urn}.

Proposition \ref{prop:condensation_urn_finite} applies to our new urn process $\wt{\mc{U}}_n$. Hence,
\begin{align*}
\mu_n([0,a])&=\wt\mu_n(0)\\
&\stackrel{a.s.}{\to} \frac12\P[f(F)=0]+\frac12\P\l[L<\infty\text{ and }\max_{i=1,\ldots,T}f(C_i)=0\r]\\
&=\frac12\P[F\leq a]+\frac12\P\l[L<\infty\text{ and }\max_{i=1,\ldots,T}C_i\leq a\r].
\end{align*}
Since $\mu_n([0,1])=1$, we have $\mu_n([0,1])\to \mu([0,1])=1$. Thus, we have $\mu_n([0,a])\stackrel{a.s.}{\to}\mu([0,a])$ for all $a\in[0,1]$.
This implies that $\int f d\mu_n\stackrel{a.s.}{\to}\int f d\mu$ for all right-continuous step functions $f:[0,1]\to\R$.
Let $C[0,1]$ denote the space of continuous functions from $[0,1]\to\R$, with the $||\cdot||_\infty$ norm.
Any $f\in C[0,1]$ can be approximated uniformly by right-continuous step functions,
so it follows that $\int f d\mu_n \stackrel{a.s.}{\to}\int f d\mu$ for all $f\in C[0,1]$.
Moreover, the sequence of measures $(\P\circ \mu_n)$ is tight, because they are measures on the compact space $[0,1]$.
Thus, the conditions for Corollary 2.2 of \cite{BPR2006} hold, with the conclusion that, almost surely, $\mu_n$ converges weakly to $\mu$.
\end{proof}

\subsection{Proof of Proposition \ref{prop:condensation_urn_finite}}
\label{sec:cond_01}

Recall that the conditions of Proposition \ref{prop:condensation_urn_finite} specify that the fitness space is a two point set $\{0,1\}$, and that each fitness occurs with positive probability. 
We assume these conditions for the duration of Section \ref{sec:cond_01}.
The first step of the proof is to show that, as $n\to\infty$,
\begin{equation}\label{eq:lim_Pcn_eq}
\P[\col(\v{c}_n)=0]\to \P\l[L<\infty\text{ and }\max_{i=1,\ldots,L}C_i=0\r]
\end{equation}
To see this, recall that Lemma \ref{lem:gw_coupling} states that $W^n_k=|\mc{W}^n_k|$ has the same distribution as a Galton-Watson process, with offspring distribution \eqref{eq:gw_offspring}, for generations $k\leq K_n$.
Let $(\hat{W}^n_k)_{k\geq 0}$ be a Galton-Watson process with this same offspring distribution, and couple $\hat{W}^n_k$ and $W^n_k$ such that $\hat{W}^n_k=W^n_k$ for all $n$ and $k\leq K_n$. Let $\hat{L}$ be the number of leaves of $(\hat{W}^n_k)$, and let $(C_i)$ be a sequence of i.i.d.~random variables, each with distribution $F$.
We note that the offspring distribution $M$, given by \eqref{eq:gw_offspring}, of $\hat{W}^n_k$ does not depend on $n$. Since $\P[M=0]\in(0,1)$, it is easily seen that
\begin{equation}\label{eq:gw_explode_or_die}
\P[\hat{L}<\infty\text{ and }\hat{W}^n_k\neq 0]
\end{equation}
does not depend on $n$ and, moreover, tends to zero as $k\to\infty$.
We note also that for all $k,n\in\N$,
\begin{align}
\P\l[\col(\v{c}_n)=0\text{ and }\mc{W}^n_k\neq\emptyset\r]
&\leq (\tfrac12)^{k/2}+\P\l[\col(\v{c}_n)=0\text{ and }|\mc{L}^n_k|\geq k/2\r]\notag\\
&\leq (\tfrac12)^{k/2}+\P[F=0]^{k/2}.\label{eq:too_many_leaves}
\end{align}
In the above, the first line follows by Lemma \ref{lem:leaf_production}, and the second line follows because $\col(\v{c}_n)=0$ when, and only when, every source ball in $\v{c}^\downarrow_n$ has colour $0$.

Let $\epsilon>0$, let $k\in\N$ be such that \eqref{eq:too_many_leaves} and \eqref{eq:gw_explode_or_die} are both $\leq\epsilon$, and let $N\in\N$ be chosen as in Lemma \ref{lem:Kn_control}. Then, for $n\geq N$ we have
\begin{align*}
\P[\col(\v{c}_n)=0]
&=\mc{O}(\epsilon)+\P\l[\max_{i=1,\ldots,L}C_i=0,\; K_n\geq k\text{ and }W^n_k=0\r]\\
&=\mc{O}(\epsilon)+\P\l[\max_{i=1,\ldots,\hat{L}}C_i=0,\; K_n\geq k\text{ and }\hat{W}^n_k=0\r]\\
&=\mc{O}(\epsilon)+\P\l[\max_{i=1,\ldots,\hat{L}}C_i=0,\; \hat{W}^n_k=0\r]\\
&=\mc{O}(\epsilon)+\P\l[\max_{i=1,\ldots,\hat{L}}C_i=0,\; \hat{L}<\infty\r]
\end{align*}
In the above, the first line follows by \eqref{eq:too_many_leaves} and Lemma \ref{lem:Kn_control}, and the observation that $\col(\v{c}_n)=0$ if and only if all leaves of $(W^n_k)$ have colour $0$.
The second line follows by the coupling of $W^n_k$ and $\hat{W}^n_k$ introduced above.
The third line follows by applying Lemma \ref{lem:Kn_control} again,
and the final line then follows by \eqref{eq:gw_explode_or_die}.
With this in hand, \eqref{eq:lim_Pcn_eq} follows because $L$ and $\hat{L}$ have the same distribution.

We now upgrade \eqref{eq:lim_Pcn_eq} into the full statement of Proposition \ref{prop:condensation_urn_finite}.
Note that the case $a=1$ of \eqref{eq:condensation_urn_finite} claims that $1\to 1$, which is true, so it remains only to prove the case $a=0$. We have
\begin{align*}
\mu_n(0)
&=\frac{1}{|\mc{U}_n|}\sum\limits_{\v{b}\in\mc{U}_n}\1_{\l(\col(\v{b})=0\r)}\\
&=\frac{|\mc{S}_n|}{|\mc{U}_n|}\frac{1}{|\mc{S}_n|}\sum\limits_{\v{b}\in\mc{S}_n}\1_{\l(\col(\v{b})=0\r)}+
\frac{|\mc{C}_n|}{|\mc{U}_n|}\stackrel{\nu_n}{\overbrace{\frac{1}{|\mc{C}_n|}\sum\limits_{\v{b}\in\mc{C}_n}\1_{\l(\col(\v{b})=0\r)}}}.
\end{align*}
Noting that $|\mc{S}_n|/|\mc{U}_n|$ and $|\mc{C}_n|/|\mc{U}_n|$ are both equal to $\frac12$, we obtain from the strong law of large numbers that the first term of the above tends (almost surely) to $\frac12\P[F=0]$, and it remains to consider the term labelled $\nu_n$.
Thus, to prove \eqref{eq:condensation_urn_finite} we must show that
\begin{equation}\label{eq:cn_conv}
\nu_n\stackrel{a.s.}{\to}\P\l[L<\infty\text{ and }\max_{i=1,\ldots,L}C_i=0\r].
\end{equation}
From \eqref{eq:lim_Pcn_eq} we already know that $\E[\nu_n]$ converges to the right hand side of the above equation so, by dominated convergence, equation \eqref{eq:cn_conv} follows if we can show that the random sequence $\nu_n$ converges almost surely to a deterministic limit $\nu$. To establish this fact we will use the `usual' machinery of stochastic approximation (c.f.~Remark \ref{rem:stoch_approx}).

Let $(\mc{F}_n)$ be the filtration generated by $(\nu_n)$.
Let $A_{n}$ be the event that the potential parents $(\v{p}^{n}_l)_{l=1}^{R_n}$ of $\v{c}_{n}$ are all distinct, and let $A_{n}^c$ denote its complement.
These potential parents are i.i.d.~uniform samples of $\mc{U}_n$ and $|\mc{U}_n|=2n+2$, so
\begin{align*}
\P[A_{n}|R_{n}=r]
=\prod\limits_{l=1}^{r-1}\l(1-\frac{l}{2n+2}\r)
\end{align*}
Thus $\P[A_{n}|R_{n}=r]\geq (1-\frac{r}{2n})^{r}$ and by Bernoulli's inequality when $2n\geq r$ we have $\P[A_{n}|R_{n}=r]\geq 1-\frac{r^2}{2n}$. We thus obtain
\begin{align}\label{eq:P_dupl}
\P[A^c_{n}]
&=
\sum\limits_{r=1}^{2n}\P[A_{n}^c|R_{n}=r]\P[R_{n}=r] + \sum\limits_{r=2n+1}^{\infty}\P[A_{n}^c|R_{n}=r]\P[R_{n}=r]\notag\\
&\leq
\sum\limits_{r=1}^{2n}\frac{r^2}{2n}\P[R_{n}=r] + \P[R_{n}\geq 2n]\notag\\
&\leq 
\frac{1}{2n}\E[R^2]+\frac{1}{2n}\E[R].
\end{align}
Note that in the final line of the above we use that $R_n$ is an independent copy of $R$, with a distribution that does not depend on $n$.

Let $M_R$ denote the moment generating function of $R_{n}$, which does not depend on $n$.
Let us write $\lambda=\P[F=0]$ for the probability that a given source balls has colour $0$.
Then,
\begin{align}
\E[\nu_{n+1}-\nu_n\|\mc{F}_n]
&=\frac{1}{|\mc{C}_{n+1}|}\l(-\nu_n+\P\l[\col(\v{c}_{n+1})=0\|\mc{F}_n\r]\r)\notag\\
&=\frac{1}{|\mc{C}_{n+1}|}\l(-\nu_n+\E\l[\1_{A_{n+1}}\sum\limits_{r=1}^\infty\1_{\{R_{n+1}=r\}}\sum\limits_{s=0}^r\binom{r}{s}\frac{1}{2^r}(\nu_n)^{r-s}(\lambda)^{s}\r]+\mc{O}(\P[A_{n+1}^c])\r)\notag\\
&=\frac{1}{|\mc{C}_{n+1}|}\l(-\nu_n+\E\l[\1_{A_{n+1}}\sum\limits_{r=1}^\infty\1_{\{R_{n+1}=r\}}\l(\frac{\nu_n}{2}+\frac{\lambda}{2}\r)^r\r]+\mc{O}\l(n^{-1}\r)\r)\notag\\
&=\frac{1}{|\mc{C}_{n+1}|}\l(-\nu_n+\sum\limits_{r=1}^\infty\P[R_{n+1}=r]\l(\frac{\nu_n}{2}+\frac{\lambda}{2}\r)^r+\mc{O}\l(n^{-1}\r)\r)\notag\\
&=\frac{1}{|\mc{C}_{n+1}|}\l(-\nu_n+M_R\l(\frac{\nu_n}{2}+\frac{\lambda}{2}\r)+\mc{O}\l(n^{-1}\r)\r)\label{eq:stoch_approx_pre}
\end{align}
In the above, to deduce the second line from the first, we condition on the number $R_{n+1}=r$ of potential parents of $\v{c}_{n+1}$, and also on the number $s$ of potential parents of $\v{c}_n$ which are source balls; then, if all these potential parents are distinct, $(\nu_n)^{r-s}(\lambda)^s$ is the probability that all potential parents of $\v{c}_{n+1}$ have colour $0$. 
We also use \eqref{eq:P_dupl}.
The third line follow from elementary calculations, and the fourth line follows by using \eqref{eq:P_dupl} again.

From \eqref{eq:stoch_approx_pre}, writing $g(\nu)=M_R(\frac{\nu+\lambda}{2})-\nu$, and noting that $|\mc{C}_n|=n+1$, $|\nu_n|\leq 1$, we obtain
$\E[\nu_{n+1}-\nu_n\|\mc{F}_n]=\frac{1}{n}\l(g(\nu_n)+\mc{O}(n^{-1})\r)$
and thus the stochastic approximate equation
$$\nu_{n+1}-\nu_n=\frac{1}{n}\l(g(\nu_n)+\xi_{n}+\mc{O}(n^{-1})\r)$$
holds with $\xi_n=\nu_{n+1}-\E[\nu_{n+1}\|\mc{F}_n]$. Since $|\xi_n|\leq 2$ and $g:[0,1]\to[0,1]$ is continuous, it follows from Corollary 2.7 of \cite{Pemantle2007} that $\nu_n$ converges almost surely to the zero set of $g$.
Recalling that $\lambda\in(0,1)$, we have $g(0)=M_R(\frac{\lambda}{2})>0$ and $g(1)=M_R(\frac{1+\lambda}{2})-1<0$.
Moreover, $g'(\nu)=\frac12 M'_R(\frac{\nu+\lambda}{2})-1$
is an increasing function, thus $g$ has at most one turning point in $[0,1]$ and hence also precisely one root in $[0,1]$. Therefore, $\nu_n$ converges almost surely to this root.
This completes the proof of Proposition \ref{prop:condensation_urn_finite}.


\section{Proof of Theorem \ref{thm:extensive_condensation_graph}}
\label{sec:extensive_condensation}

In this section we prove Theorem \ref{thm:extensive_condensation_graph}, which asserts that
extensive condensation occurs in the model.
We assume the conditions of this theorem for the duration of Section \ref{sec:extensive_condensation};
in particular that $\E[R]>2$ with $\E[R^2]<\infty$.
From now on, we will write
\begin{align}
\zeta=\frac{\E[R]}{2},
\quad\quad\quad\quad
\beta=\frac{\E[R]-2}{\E[R]}=\frac{\zeta-1}{\zeta}.\label{eq:zeta_beta_def}
\end{align}
Note that $\zeta\in(1,\infty)$ and $\beta\in(0,1)$.
We will introduce a third variable $\xi\in(\zeta,\infty)$ that also depends only on the distribution of $R$, in Lemma \ref{lem:Gk2_upper}.

We use the following extensions of Landau notation. 
If $a_{k,n}$ and $b_{k,n}$ are a pair of doubly indexed strictly positive (real-valued) sequences, defined for all $k,n\in\N$ such that $k\leq n$, then
\begin{align*}
a_{k,n}\lesssim b_{k,n}&\;\text{ means that }\;\limsup_{k,n\to\infty} \frac{a_{k,n}}{b_{k,n}}\leq 1,\\
a_{k,n}\gtrsim b_{k,n}&\;\text{ means that }\;\liminf_{k,n\to\infty}\frac{b_{k,n}}{a_{k,n}}\geq 1,\\
a_{k,n}\sim b_{k,n}&\;\text{ means that }\;\lim_{k,n\to\infty}\frac{a_{k,n}}{b_{k,n}}=1.
\end{align*}
Note that $\lesssim,\gtrsim$ and $\sim$ do not explicitly specify which pair of variables ($k,n$ above) are to be used in the limit, but this should be clear from the context in all cases.
Our requirement for this notation comes from Lemma \ref{lem:key_asym}, which provides a key two-variable asymptotic that will be used within Section \ref{sec:ancestry_branching}.
We use the same notation for sequences $a_n,b_n$ of a single variable, with the same meaning, including when we take $k=k_n$ dependent on $n$.

\subsection{Proof of Theorem \ref{thm:extensive_condensation_graph}}
\label{sec:ext_cond_proof}

The proof of Theorem \ref{thm:extensive_condensation_graph} relies 
on behaviour within the critical window $[cn^\beta, Cn^\beta]$, 
where $c$ is suitably small and $C$ is suitably large.
We will show that the fittest source ball born during this window has a non-negligible expected family size at time $n$, as $n\to\infty$.
This, in turn, will be proved by showing that $\v{c}_n$ has non-vanishing probability to be descended from the fittest source ball within 
$\{\v{s}_{cn^\beta},\ldots,\v{s}_{n^\beta}\}$.
Note that such a ball has a non-vanishing probability to be the fittest ball within 
$\{\v{s}_{0},\ldots,\v{s}_{Cn^\beta}\}$.

\begin{remark}\label{rem:Cinteger}
We assume without loss of generality that $Cn^\beta$ and $cn^\beta$ are integer. This can be achieved by adding a small quantity, at most $n^{-\beta}$, to $c,C$. The difference is sufficiently small that it does not change our arguments, so we continue to regard $c,C$ as fixed constants, independent of $n$.
\end{remark}

Let $\v{s}_{k(n)}$ denote the fittest source ball in $\{\v{s}_{0},\ldots,\v{s}_{Cn^\beta}\}$.
The proof of Theorem \ref{thm:extensive_condensation_graph} has two key ingredients.
The first, Proposition \ref{prop:ancestry_stablization}, will be used to show that $\P[\v{s}_{k(n)}\in\v{c}^\downarrow_n]$ is bounded away from $0$ as $n\to\infty$. 
The second, Proposition \ref{prop:branching_phase_pickup}, shows that during time $[Cn^\beta,n]$ not many source balls are included in the genealogy of $\v{c}^\downarrow_n$; few enough that none of them are likely to be fitter than $\v{s}_{k(n)}$.
Let us now state these two results rigorously, for which we require some notation.

Consider balls that are potential ancestors of $\v{c}^\downarrow_n$.
Each such ball has a natural multiplicity associated to it:
the number of times it was chosen as a potential parent of some (other) potential ancestor of $\v{c}^\downarrow_n$.
Thus, counting with multiplicity means we are ignoring coalescences, within the genealogy of $\v{c}^\downarrow_n$.
For $k\leq n$ define
\begin{equation}\label{eq:Nk_def}
N^n_{k}=\sum\limits_{j=k+1}^n \1_{\l(\v{c}_j\in\v{c}_n^\downarrow\r)}\sum\limits_{l=1}^{R_j}\1_{\l(\v{p}_{j,l}=\v{s}_k\r)}
\end{equation}
to be the number of times that $\v{s}_k$ is a potential ancestor of $\v{c}_n$, counted with multiplicity.
For $i\leq i'\leq n$,
\begin{equation}\label{eq:Nkip_def}
N^n_{i,i'}=\sum\limits_{k=i}^{i'} N^n_{k}.
\end{equation}
In words, $N^n_{i,i'}$ is the number of source balls $\{\v{s}_{i},\ldots,\v{s}_{i'}\}$, counted with multiplicity, that are potential ancestors of $\v{c}_n$.
Thus, $|\v{c}^\downarrow_n\cap\{\v{s}_{i},\ldots,\v{s}_{i'}\}|\leq N^n_{i,i'}$.
Define also
\begin{align}
\mc{H}^n_k&=\{\v{u}\in\mc{U}_{k}\-\v{u}\in\mc{P}_j\text{ for some }\v{c}_j\in\v{c}^\downarrow_n\text{ with }j\geq k+1\}\notag\\
H^n_{k}&=|\mc{H}^n_k|\label{eq:H_def}
\end{align}
Note that $\mc{H}^n_k$ is the set of balls that were born (non-strictly) before time $k$,
and were a potential parent of some $\v{c}_j\in\v{c}^\downarrow_n$ where $j> k$.
The quantity $H^n_{k}$ counts such balls \textit{without} multiplicity.
We note that the urn contains $2l+2$ balls at time $l$, so $\frac{1}{2l+2}H^n_{l}$ 
represents the fraction of the urn included in $\mc{H}^n_{l}$ at time $l$.

\begin{prop}\label{prop:ancestry_stablization}
For all $C>\xi^{1/\zeta}$ and $c\in(0,2)$ we have
$\E\l[\inf\limits_{l=cn^\beta,\ldots,2n^\beta}\frac{1}{2l+2}H^n_{l}\r]\gtrsim \frac{\zeta^2}{4\xi}\frac{\zeta}{\zeta+4^{\zeta}e^{c/C}}.$
\end{prop}

\begin{prop}\label{prop:branching_phase_pickup}
For all $C\in(0,\infty)$ we have
$\E[N^n_{Cn^\beta,n}]\lesssim\frac{n^\beta}{\beta C^{\zeta-1}}.$
\end{prop}
In both propositions, the asymptotic inequality is understood to apply as $n\to\infty$.
We will give the proof of Theorem \ref{thm:extensive_condensation_graph} now, subject to these two propositions. 
We will then prove Proposition \ref{prop:branching_phase_pickup} in Section \ref{sec:ancestry_branching}, and Proposition \ref{prop:ancestry_stablization} in Section \ref{sec:ancestry_stablization}.

\begin{proof}[Of Theorem \ref{thm:extensive_condensation_graph}, subject to Propositions \ref{prop:ancestry_stablization} and \ref{prop:branching_phase_pickup}.]
Let $c,C$ satisfy $0<c<C<\infty$ with $c<2$ and $C>\xi^{1/\zeta}$, with precise values to be chosen later.
Recall that $\v{s}_{k(n)}$ denotes the (almost surely unique) fittest source ball within $\{\v{s}_0,\ldots,\v{s}_{Cn^\beta}\}$,
and let 
$$S_n=|\v{s}_{k(n)}^\uparrow\cap\mc{U}_n|$$ 
denote the size of the family of $\v{s}_{k(n)}$ at time $n$.
Let $Q^{j,n}$ be the event that all sources in $\{\v{s}_{l}\-l=Cn^\beta+1,\ldots,j-1\}\cap\v{c}^\downarrow_j$ are less fit than $\v{s}_{k(n)}$.
Note that
\begin{align}
\E[S_n]
&=1+\sum\limits_{j=0}^n\P\l[\{\v{s}_{k(n)}\in\v{c}^\downarrow_j\} \cap Q^{j,n}\r]\notag\\
&\geq\sum\limits_{j=2^{-1/\beta}n}^{n}\P\l[\{\v{s}_{k(n)}\in\v{c}^\downarrow_j\} \cap Q^{j,n}\r].\label{eq:Ln}
\end{align}
In the above, on the final line, the summation includes $j\in\N$ such that $2^{-1/\beta}n\leq j\leq n$.
Consider $n$ large enough that $Cn^\beta<2^{-1/\beta}n$, and take such a $j$. Then
\begin{align}
\P\l[\{\v{s}_{k(n)}\in\v{c}^\downarrow_j\} \cap Q^{j,n}\r]
&\geq \P\l[\v{s}_{k(n)}\in\v{c}^\downarrow_j\r]-\P\l[\Omega\sc Q^{j,n}\r]\label{eq:skncjQ}
\end{align}
We now handle the first term on the right of \eqref{eq:skncjQ}.
Let $P_n$ be the event that $cn^\beta\leq k(n)\leq n^\beta$.
We have
\begin{align}
\P\l[\v{s}_{k(n)}\in\v{c}^\downarrow_j\r]\notag
&=\P\l[\v{s}_{k(n)}\in\v{c}^\downarrow_j\,\Big|\,P_n\r]\P[P_n]\notag\\
&=\P\l[\v{s}_{k(n)}\in \mc{H}^j_{k(n)}\,\Big|\,P_n\r]\P[P_n]\notag.
\end{align}
Note that $k(n)$ is uniform on $\{0,1,\ldots,Cn^\beta\}$ and measurable with respect to the fitness values $(F_i)$.
The event $P_n$ is also measurable with respect to $(F_i)$.
These fitness values are independent of the sampling of potential parents, hence $\mc{H}^j_l$ and $H^j_l$ are independent of $k(n)$ and $P_n$.
Moreover, for all $j\geq i>l$, the potential parents of $\v{c}_{i}\in\v{c}^\downarrow_j$ were sampled uniformly and i.i.d.~from $\mc{U}_{i}$; conditional on any such potential parent $\v{p}$ being within $\mc{U}_l$, the distribution of $\v{p}$ is uniform on $\mc{U}_l$.
Thus, the conditional distribution of $\mc{H}^j_l$ given $H^j_l$ is also uniform (on the subsets of $\mc{U}_l$ that have size $H^j_l$).
Consequently, 
\begin{align}
\P\l[\v{s}_{k(n)}\in\v{c}^\downarrow_j\r]&=\E\l[\frac{1}{2k(n)+2}H^j_{k(n)}\,\Big|\,P_n\r]\P[P_n]\notag\\
&\geq \E\l[\inf_{l=cn^\beta,\ldots,n^\beta}\frac{1}{2l+2}H^j_{l}\r]\P[P_n]\notag\\
&\geq \E\l[\inf_{l=cj^\beta,\ldots,2j^\beta}\frac{1}{2l+2}H^j_{l}\r]\P[P_n]\label{eq:kncj_lower_pre}
\end{align}
Note that the third line above follows because $cj^\beta\leq cn^\beta$ and $n^\beta\leq 2j^\beta$.
We now apply Proposition \ref{prop:ancestry_stablization}, which gives that there exists $N\in\N$ such that
for all $n\geq N$, 
$$
\E\l[\inf_{l=cj^\beta,\ldots,2j^\beta}\frac{1}{2l+2}H^j_{l}\r]
\geq 
\frac12\frac{\zeta^2}{4\xi}\frac{\zeta}{\zeta+4^{\zeta}e^{c/C}}.
$$
Note that $\P[P_n]\sim\frac{1-c}{C}$.
Increasing $N$ if necessary, we may also assume that $\P[P_n]\geq\frac12\frac{1-c}{C}$ for $n\geq N$.
Thus, continuing from \eqref{eq:kncj_lower_pre},
\begin{equation}
\P\l[\v{s}_{k(n)}\in\v{c}^\downarrow_j\r]
\geq \frac{1}{4}\frac{\zeta^2}{4\xi}\frac{\zeta}{\zeta+4^{\zeta}e^{c/C}}\frac{1-c}{C}.\label{eq:kncj_lower}
\end{equation}
We now look to control the second term on the right of \eqref{eq:skncjQ}. 
The statement of Theorem \ref{thm:extensive_condensation_graph} relates to a property of the graph $(\mc{G}_n)$,
and consequently in view of Remark \ref{rem:fitness_Gn} we may assume that the fitness values are sampled according to
the uniform distribution on $[0,1]$.
By definition of $k(n)$, $\col(\v{s}_{k(n)})$ has the same distribution as $\max\{U_0,\ldots,U_{Cn^\beta}\}$ where the $(U_i)$ are i.i.d.~uniform random variables on $[0,1]$. It follows that for any $a>0$,
\begin{equation}\label{eq:colvskn}
\P\l[\col\l(\v{s}_{k(n)}\r)< 1-\frac{1}{an^\beta}\r]=\P\l[U_1< 1-\frac{1}{an^\beta}\r]^{Cn^\beta}=\l(1-\frac{1}{an^\beta}\r)^{Cn^\beta}\leq e^{-C/a}.
\end{equation}
The final inequality on the right hand side of the above follows from recalling that $(1-\frac{b}{x})^{x}\leq e^{-b}$ for all $0<b\leq x$.
Choosing $a=\frac{C}{\zeta\log C}$, we have
\begin{align}
\P\l[\Omega\sc Q^{j,n}\r]
&\leq \P\l[\l(\Omega\sc Q^{j,n}\r)\cap\l\{\col\l(\v{s}_{k(n)}\r)\geq 1-\frac{1}{an^\beta}\r\}\r]+\P\l[\col\l(\v{s}_{k(n)}\r)< 1-\frac{1}{an^\beta}\r]\notag\\
&\leq \frac{1}{an^\beta}\E[N^j_{Cn^\beta,j}]+e^{-C/a}\notag\\
&\leq \frac{1}{an^\beta}\E[N^j_{Cj^\beta,j}]+e^{-C/a}\label{eq:N_control}.
\end{align}
In the above, to deduce the second line, for the second term we use \eqref{eq:colvskn};
for the first term we recall $N^j_{Cn^\beta,j}$ 
is the number of source balls (counted with multiplicity) within $\v{c}^\downarrow_j$,
that were born between $Cn^\beta$ and $j$, and that each such source ball samples its fitness independently and uniform on $[0,1]$.
The third line follows trivially from the second.

We are now in a position to apply Proposition \ref{prop:branching_phase_pickup} by which, increasing $N$ again if necessary, for all $j\geq 2^{-1/\beta}N$ we may assume
$\E[N^j_{Cj^\beta,j}]\leq 2\frac{j^\beta}{\beta C^{\beta-1}}.$
Recalling that $2^{-1/\beta}n\leq j\leq n$, we thus obtain that when $n\geq N$
\begin{align}
\P\l[\Omega\sc Q^{j,n}\r]
&\leq 2\frac{1}{a\beta C^{\zeta-1}}\l(\frac{j}{n}\r)^\beta+e^{-C/a}\notag\\
&\leq \frac{2\zeta\log C}{\beta C^{\zeta}}+\frac{1}{C^\zeta}.\label{eq:Q_upper}
\end{align}
Putting \eqref{eq:kncj_lower} and \eqref{eq:Q_upper} into \eqref{eq:skncjQ}, and then putting \eqref{eq:skncjQ} into \eqref{eq:Ln} we obtain that
\begin{equation}\label{eq:Lnn}
\E[S_n]\gtrsim \l(1-2^{-1/\beta}\r)n\l[\frac{1}{4}\frac{\zeta^2}{4\xi}\frac{\zeta}{\zeta+4^{\zeta}e^{c/C}}\frac{1-c}{C}-\frac{2\zeta\log C}{\beta C^{\zeta}}-\frac{1}{C^\zeta}\r].
\end{equation}
Noting that $\zeta>1$ and $\beta\in(0,1)$, we may choose $c=\frac12$ and $C$ sufficiently large that the term in square brackets, in the above equation, is strictly positive.
We thus obtain that $\frac{1}{n}\E[S_n]\gtrsim \gamma$, where $\gamma>0$ is equal to the right hand side of \eqref{eq:Lnn} divided by $n$.

Let $\epsilon>0$ and recall $\ell_n$ from \eqref{eq:ell_def}. Since $\v{s}_{k(n)}$ is the initial half-edge of the fittest of the first $Cn^\beta+1$ vertices, it is clear that $F_{k(n)}=\col(\v{s}_{k(n)})\to 1$ almost surely as $n\to\infty$. Hence we may choose $N\in\N$ such that for all $n\geq N$, $\P[F_{k(n)}\geq 1-\epsilon]\geq 1-\epsilon$, and when this event occurs we have
$\ell_n([1,1-\epsilon])\geq \frac{1}{2(n+1)}S_n$.
Thus $\E[\ell_n([1,1-\epsilon])]\gtrsim (1-\epsilon)\frac{\gamma}{2}$ as $n\to\infty$, which implies that extensive condensation occurs.
\end{proof}

\begin{remark}
\label{rem:ext_vs_non}
In the above proof, it is crucial that, within \eqref{eq:Lnn}, the first term inside the square brackets (which came from Proposition \ref{prop:ancestry_stablization}) has order $C^{-1}$ and the second term (which came from Proposition \ref{prop:branching_phase_pickup}) has the lower order $C^{-\zeta}$. 
Let us briefly attempt to explain why this occurs.
Looking backwards in time, as $k$ decreases, through the genealogy of $\v{c}^\downarrow_n$, there are two key transitions that take place:
\begin{enumerate}
\item The point at which $\mc{H}^n_k$ grows large enough to include source vertices that were unusually fit for their time of birth.
\item The point at which $\mc{H}^n_k$ grows large enough to include individual source vertices with non-vanishing probability.
\end{enumerate}
A priori, these two transitions might not occur simultaneously, although it is clear that 1 must come non-strictly before 2.
The fact that these two transitions do occur simultaneously is what leads to extensive condensation in PAC -- Propositions \ref{prop:ancestry_stablization} and \ref{prop:branching_phase_pickup} (the latter via \eqref{eq:N_control}) tell us that both happen during the critical window.
If these two transitions did not occur simultaneously, one should expect non-extensive condensation.

We believe that, as a rule of thumb, the standard stochastic approximation theorems (see e.g.~Section 2.4 of \cite{Pemantle2007}) are not suitable to prove extensive condensation, in the absence of a persistent hub or other simplifying factor. The reason is that stochastic approximation determines only if convergence (of some well chosen quantity to a suitable limit) occurs -- it does not identify the rate of convergence. Thus stochastic approximation alone will not identify the asymptotic times at which the two key transitions above take place, and thus does not distinguish between extensive and non-extensive condensation. 
\end{remark}

\begin{remark}
The reader may ask why we focus on $\liminf_n \E[\frac{1}{n}S_n]$ rather than $\E[\liminf_n \frac{1}{n}S_n]$. 
Firstly, as discussed in Section \ref{sec:background}, it is $\liminf_n \E[\cdot]$ which relates to the existing meaning of extensive condensation within the literature.
Secondly, for PAC, we expect that $\E[\liminf_n \frac{1}{n}S_n]=0$. We expect this because the critical window $[cn^\beta,Cn^\beta]$ will, infinitely often as $n\to\infty$, contain a large number $M$ of vertices that each became a newly fittest vertex at their time of birth. Let us call this event $E_n$. When $E_n$ happens and $M$ is large we believe that $S_n$ is close to zero, essentially because 
the many unusually fit vertices born during the critical window will then compete with each other as they grow, resulting in a situation where the largest vertex at time $n$ has degree $\epsilon_M n$, where $\epsilon_M$ is close to zero when $M$ is large.
Note that $\E[S_n]$ can remain bounded away from zero only because $\P[E_n]\to 0$. We do not attempt a rigorous statement or proof of these claims within the present article.
\end{remark}

\subsection{The branching phase of the genealogy}
\label{sec:ancestry_branching}

We now turn our attention to proving Propositions \ref{prop:ancestry_stablization} and \ref{prop:branching_phase_pickup}.
The latter will appear at the end of Section \ref{sec:ancestry_branching}, and the former in Section \ref{sec:ancestry_stablization}. 
In both cases it is key to investigate the growth of the genealogy $\v{c}^\downarrow_n$, backwards in time.
In this section we fix a cue ball $\v{c}^\downarrow_n$, and look backwards in time at the period during which its genealogy is dominated by branching.
We analyse this phase of the genealogy using iterative methods, with each iteration moving one step further backwards in time.
The following lemma will play a key role in the calculations.
\begin{lemma}
\label{lem:key_asym}
Let $\alpha\geq 0$ and $\gamma_j\in\R$ such that $\sum_j|\gamma_j|<\infty$. Then as $k,n\to\infty$ with $k\leq n$,
$$\prod_{j=k}^n(1+\frac{\alpha}{j}+\gamma_j)\sim\l(\frac{n}{k}\r)^\alpha.$$
\end{lemma}
A proof of this lemma is given in Appendix \ref{app:key_asym}.
We will also make regular use of the following elementary inequality: for $j\in\N$ and $x\in[0,1]$,
\begin{align}
1-jx\leq (1-x)^j \leq 1-jx+\binom{j}{2}x^2 \label{eq:taylor}.
\end{align}

We now define the notation that we will use to explore $\v{c}^\downarrow_n$ backwards in time.
Recall that the potential parents $\mc{P}_j=\{\v{p}_{j,1},\ldots,\v{p}_{j,R_j}\}$ of $\v{c}_j$ are i.i.d.~samples from $\mc{U}_{j-1}$.
For $k=0,1,\ldots,n$ we define
\begin{align}
G^n_k&=\sum\limits_{j=k}^n \1_{\l(\v{c}_j\in\v{c}_n^{\downarrow}\r)} \sum\limits_{l=1}^{R_j}\1_{\l(\v{p}_{j,l}\in\mc{U}_{k-1}\r)}\label{eq:Gk_def}
\end{align}
In words, $G_k^n$ counts, with multiplicity, potential parents $\v{p}$ of $\{\v{c}_{k},\ldots,\v{c}_n\}\cap \v{c}_n^\downarrow$ that were born strictly before time $k$.
Note that, as usual, $\cdot^n$ denotes a superscript $n$ and not an exponent.

Our first goal in this section is to find upper and lower bounds for $\E[G^n_k]$.
In order to establish the lower bound we will also require an upper bound on $\E[(G^n_k)^2]$.
We end with two applications of these bounds: in Lemma \ref{lem:branching_end_control} we show that when $k\approx n^\beta$ we have $\E[G^n_k]\approx k$ and, with this choice of $k$ we give the proof of \ref{prop:branching_phase_pickup}.
Let
\begin{align}
A_k^n&=\sum\limits_{j=k+1}^n \1_{\l(\v{c}_j\in\v{c}_n^{\downarrow}\r)} \sum\limits_{l=1}^{R_j}\1_{\l(\v{p}_{j,l}\in\{\v{c}_{k},\v{s}_{k}\}\r)}\label{eq:Gk_loss}
\end{align}
In words, $A^n_k$ is the number of times (counted with multiplicity) that either $\v{c}_{k}$ or $\v{s}_{k}$ is chosen as a potential parent of some $\v{c}\in\{\v{c}_{k+1},\ldots,\v{c}_n\}\cap \v{c}_n^\downarrow$. Similarly, let
\begin{align}
B^n_k&=\1_{\l(\v{c}_{k}\in\v{c}_n^{\downarrow}\r)} R_{k}\label{eq:Gk_gain}\\
&=\1_{\l(\v{c}_{k}\in\v{c}_n^{\downarrow}\r)}\sum_{l=1}^{R_{k}}\1_{\l(\v{p}_{k,l}\in\mc{U}_{k-1}\r)}\label{eq:Gk_gain_2}
\end{align}
In words, $B^n_k$ is the number of potential parents (counted with multiplicity) of $\v{c}_{k}$ when $\v{c}_{k}$ is itself in $\v{c}_n^\downarrow$, and is zero otherwise. Note that all such potential parents are automatically elements of $\mc{U}_{k-1}$, justifying \eqref{eq:Gk_gain_2}.
It is immediate from \eqref{eq:Gk_def}, \eqref{eq:Gk_loss} and \eqref{eq:Gk_gain_2} that
\begin{equation}\label{eq:Gm_recursion_pre}
G^n_{k}=G^n_{k+1}-A^n_k+B^n_k.
\end{equation}

 For $k=0,1,2,\ldots$, we define the sequence of decreasing $\sigma$-fields
$$\mathscr{G}_k=\sigma\l(R_j, \1_{\l(\v{c}_j=\v{p}_{i,l}\r)}, \1_{\l(\v{s}_j=\v{p}_{i,l}\r)}\-i\geq j\geq k,\; l\in\N\r)$$
In words, $\mathscr{G}_k$ contains the information of:
the number $R_j$ of potential parents of each of the balls $\{\v{c}_{k},\v{c}_{k+1},\ldots\}\cup\{\v{s}_{k},\v{s}_{k+1},\ldots\}$,
plus
the identities of these potential parents in cases where they are also elements of $\{\v{c}_{k},\v{c}_{k+1},\ldots\}\cup\{\v{s}_{k},\v{s}_{k+1},\ldots\}$.

We will take conditional expectation of \eqref{eq:Gm_recursion_pre} with respect to $\mathscr{G}_{k+1}$ in Lemma \ref{lem:Gk_1st}, and the same for $(G^n_k)^2$ in Lemma \ref{lem:Gk2_upper}.
To this end, we note that:
\begin{enumerate}
\item[($\star$)]
If $k+1\leq j\leq n$, and $1\leq l\leq R_j$, then
$\1_{\l(\v{c}_j\in\v{c}^\downarrow_n\r)}$
and
$\1_{\l(\v{p}_{j,l}\in\mc{U}_{k}\r)}$
are both $\mathscr{G}_{k+1}$ measurable.
\end{enumerate}
The first claim holds because if there is a potential ancestral line connecting $\v{c}_n$ to $\v{c}_j$ then $\mathscr{G}_{k+1}$ can see the identities of these ancestors. The second holds because $\v{p}_{j,l}\in\mc{U}_{k}$ if and only if $\v{p}_{j,l}$ was not born after time $k+1$.

We record one further observation for future use:
\begin{itemize}
\item[($\dagger$)] Consider $k+1\leq j\leq n$, and $1\leq l\leq R_j$. On the event that $\v{p}_{j,l}\in\mc{U}_{k}$, we have that $\v{p}_{j,l}$ is uniformly distributed on $\mc{U}_{k}$, with distribution independent of $\mathscr{G}_{k+1}$.
\end{itemize}
This observation is an immediate consequence of the fact that each potential parent $\v{p}_{j,l}$ of $\v{c}_j$ is sampled, independently of all else, uniformly from the set $\mc{U}_{j-1}$.

\begin{lemma}\label{lem:Gk_1st}
For $k<n$, we have
\begin{enumerate}
\item $\textstyle\E[A^n_k\|\mathscr{G}_{k+1}]=\frac{G^n_{k+1}}{k+1},$ 
\item $\textstyle\E\l[G^n_k\|\mathscr{G}_{k+1}\r]=G^n_{k+1}-\frac{G^n_{k+1}}{k+1}+2\zeta\l(1-\l(1-\frac{1}{2(k+1)}\r)^{G^n_{k+1}}\r).$
\end{enumerate}
\end{lemma}
\begin{proof}
Recall from \eqref{eq:zeta_beta_def} that $\zeta=\frac{\E[R]}{2}$.
We address the two claims in turn.
In short, the first claim holds because, by ($\dagger$), each ball within the subset of $\mc{U}_{k}$ counted by $G^n_{k+1}$ has chance $\frac{2}{|\mc{U}_k|}=\frac{1}{k+1}$ of being an element of $\{\v{c}_k,\v{s}_k\}$.
Formally, from \eqref{eq:Gk_loss}, since $\v{c}_{k},\v{s}_{k}\in\mc{U}_{k}$ we have that
$$A^n_k=\sum\limits_{j=k+1}^n \1_{\l(\v{c}_j\in\v{c}_n^\downarrow\r)}
\sum\limits_{l=1}^{R_j}\1_{\l(\v{p}_{j,l}\in\mc{U}_{k}\r)}\1_{\l(\v{p}_{j,l}\in\{\v{c}_{k},\v{s}_{k}\}\r)}.$$
Taking conditional expectation with respect to $\mathscr{G}_{k+1}$,
\begin{align}
\E[A_k^n\|\mathscr{G}_{k+1}]
&=\sum\limits_{j=k+1}^n \1_{\l(\v{c}_j\in\v{c}_n^\downarrow\r)}\sum\limits_{l=1}^{R_j}\1_{\l(\v{p}_{j,l}\in\mc{U}_{k}\r)}\E\l[\1_{\l(\v{p}_{j,l}\in\{\v{c}_{k},\v{s}_{k}\}\r)}\,\Big|\,\mathscr{G}_{k+1}\r]\notag\\
&=\frac{1}{k+1}\sum\limits_{j=k+1}^n \1_{\l(\v{c}_j\in\v{c}_n^\downarrow\r)}\sum\limits_{l=1}^{R_j}\1_{\l(\v{p}_{j,l}\in\mc{U}_{k}\r)}\label{eq:Cj_E}
\end{align}
Here, to deduce the first line we use $(\star)$, and to deduce the second line we use $(\dagger)$.
The first claim now follows from \eqref{eq:Gk_def}.

We now address the second claim.
We will take conditional expectation of \eqref{eq:Gm_recursion_pre} with respect to $\mathscr{G}_{k+1}$. By $(\star)$, $G^n_{k+1}$ is $\mathscr{G}_{k+1}$ measurable. We have already calculated $\E[A^n_k\|\mathscr{G}_{k+1}]$ above and it remains to calculate $\E[B^n_k\|\mathscr{G}_{k+1}]$.
From \eqref{eq:Gk_gain} we have
\begin{align}
B^n_k
&=R_{k}\l(1-\1_{\l(\v{c}_{k}\notin\v{c}_n^\downarrow\r)}\r)\notag\\
&=R_{k}\l(1-\prod\limits_{j=k+1}^n\prod\limits_{l=1}^{R_j}\l[\1_{\l(\v{c}_j\in\v{c}^\downarrow_n\r)}\1_{\l(\v{p}_{j,l}\neq\v{c}_{k}\r)}+\1_{\l(\v{c}_j\notin\v{c}^\downarrow_n\r)}\r]\r)\notag\\
&=R_{k}\l(1-\prod\limits_{j=k+1}^n\prod\limits_{l=1}^{R_j}\l[\1_{\l(\v{c}_j\in\v{c}^\downarrow_n\r)}\l[\1_{\l(\v{p}_{j,l}\in\mc{U}_{k}\r)}\1_{\l(\v{p}_{j,l}\neq\v{c}_{k}\r)}+\1_{\l(\v{p}_{j,l}\notin\mc{U}_{k}\r)}\r]+\1_{\l(\v{c}_j\notin\v{c}^\downarrow_n\r)}\r]\r).\label{eq:Gk_gain_inter_pre}
\end{align}
Therefore,
\begin{align}
\E&[B^n_k\|\mathscr{G}_{k+1}]\notag\\
&=2\zeta\l(1-\prod\limits_{j=k+1}^n\prod\limits_{l=1}^{R_j}\l[\1_{\l(\v{c}_j\in\v{c}^\downarrow_n\r)}\l[\1_{\l(\v{p}_{j,l}\in\mc{U}_{k}\r)}\E\l[\1_{\l(\v{p}_{j,l}\neq\v{c}_{k}\r)}\,\Big|\,\mathscr{G}_{k+1}\r]+\1_{\l(\v{p}_{j,l}\notin\mc{U}_{k}\r)}\r]+\1_{\l(\v{c}_j\notin\v{c}^\downarrow_n\r)}\r]\r)\notag\\
&=2\zeta\l(1-\prod\limits_{j=k+1}^n\prod\limits_{l=1}^{R_j}\l[\1_{\l(\v{c}_j\in\v{c}^\downarrow_n\r)}\l[\1_{\l(\v{p}_{j,l}\in\mc{U}_{k}\r)}\l(1-\frac{1}{|\mc{U}_{k}|}\r)+\1_{\l(\v{p}_{j,l}\notin\mc{U}_{k}\r)}\r]+\1_{\l(\v{c}_j\notin\v{c}^\downarrow_n\r)}\r]\r)\notag\\
&=2\zeta\l(1-\l(1-\frac{1}{2(k+1)}\r)^{G^n_{k+1}}\r). \label{eq:Gk_gain_inter}
\end{align}
Here, in the first line we use $(\star)$ and the fact that $R_{k}$ is independent of $\mathscr{G}_{k+1}$, with mean $\E[R]=2\zeta$.
We use $(\dagger)$ to deduce the second line, and the final line then follows from \eqref{eq:Gk_def} and $|\mc{U}_{k}|=2k+2$.
The stated result follows.
\end{proof}

\begin{lemma}\label{lem:Gk_1st_upper}
As $k,n\to\infty$ with $k\leq n$ we have
$\E[G^n_k]\lesssim 2\zeta \l(\frac{n}{k}\r)^{\zeta-1}$.
\end{lemma}
\begin{proof}
From Lemma \ref{lem:Gk_1st} and the left hand side of \eqref{eq:taylor},
\begin{align*}
\E[G^n_{k}]
&\leq \E[G^n_{k+1}]-\frac{\E[G^n_{k+1}]}{k+1}+2\zeta\frac{\E[G^n_{k+1}]}{2k+2}
=\E[G^n_{k+1}]\l(1+\frac{\zeta-1}{k+1}\r).
\end{align*}
By iterating the above inequality we obtain that
$
\E[G^n_k]
\leq\E[G^n_n] \prod_{j=k+1}^n\l(1+\frac{\zeta-1}{j}\r).
$
The result follows by applying Lemma \ref{lem:key_asym} and noting that $G^n_n=R_n$, with expectation $2\zeta$.
\end{proof}

\begin{lemma}\label{lem:Gk2_upper}
As $k,n\to\infty$ with $k\leq n$ we have
$\E[(G^n_k)^2]\lesssim \xi\l(\frac{n}{k}\r)^{2(\zeta-1)}$, where $\xi\in(\zeta^2,\infty)$ is a constant that depends only on the distribution of $R$.
\end{lemma}
\begin{proof}
We will first show that for all $k\leq n$,
\begin{equation}
\label{eq:Ak2_E}
\E\l[(A^n_k)^2\|\mathscr{G}_{k+1}\r]\leq \frac{G^n_{k+1}}{k+1}+\l(\frac{G^n_{k+1}}{k+1}\r)^2.
\end{equation}
To this end, let
$C^k_j=\sum_{l=1}^{R_j}\1_{\l(\v{p}_{j,l}\in\{\v{c}_k,\v{s}_k\}\r)}$.
In similar style to \eqref{eq:Cj_E} note that when $k<j$,
\begin{align}
\E[C^k_j\|\mathscr{G}_{k+1}]
&=\sum\limits_{l=1}^{R_j}\1_{\l(\v{p}_{j,l}\in\mc{U}_k\r)}\E\l[\1_{\l(\v{p}_{j,l}\in\{\v{c}_k,\v{s}_k\}\r)}\;\Big|\;\mathscr{G}_{k+1}\r]\notag\\
&=\frac{1}{k+1} \sum\limits_{l=1}^{R_j}\1_{\l(\v{p}_{j,l}\in\mc{U}_k\r)}.\label{eq:Ckj_E}
\end{align}
Here, the first equality follows by $(\star)$ and the second by $(\dagger)$.

We now aim to deduce \eqref{eq:Ak2_E}.
From \eqref{eq:Gk_loss} we have
\begin{align*}
(A^n_k)^2=A^n_k+2\sum\limits_{j=k+1}^n\sum\limits_{j'=k+1}^{j-1}\1_{\l(\v{c}_j\in\v{c}_n^\downarrow\r)}\1_{\l(\v{c}_{j'}\in\v{c}_n^\downarrow\r)}C^k_{j}C^k_{j'}.
\end{align*}
For $k+1\leq j'<j\leq n$ we have in particular that $j'\neq j$, hence $\v{p}_{j,l}$ and $\v{p}_{j',l'}$ are independent of each other.
Hence also $C^n_j$ and $C^n_{j'}$ are also independent, and in particular
$
\E\l[C^k_jC^k_{j'}\|\mathscr{G}_{k+1}\r]
=\E\l[C^k_j\|\mathscr{G}_{k+1}\r]\E\l[C^k_{j'}\|\mathscr{G}_{k+1}\r].
$
Therefore,
\begin{align*}
\E[(A^n_k)^2\|\mathscr{G}_{k+1}]
&=\E[A^n_k\|\mathscr{G}_{k+1}]+2\sum\limits_{j=k+1}^n\sum\limits_{j'=k+1}^{j-1}\1_{\l(\v{c}_j\in\v{c}_n^\downarrow\r)}\1_{\l(\v{c}_{j'}\in\v{c}_n^\downarrow\r)}\E\l[C^k_{j}\|\mathscr{G}_j\r]\E\l[C^k_{j'}\|\mathscr{G}_{j'}\r]\\
&\leq \frac{G^n_{k+1}}{k+1}+\frac{1}{(k+1)^2}\l(\sum\limits_{j=k+1}^n\1_{\l(\v{c}_j\in\v{c}_n^\downarrow\r)}\sum\limits_{l=1}^{R_j}\1_{\l(\v{p}_{j,l}\in\mc{U}_{k}\r)}\r)^2\\
&=\frac{G^n_{k+1}}{k+1}+\frac{(G^n_{k+1})^2}{(k+1)^2}.
\end{align*}
Here, the second line follows by Lemma \ref{lem:Gk_1st} and \eqref{eq:Ckj_E}.
The final line then follows from \eqref{eq:Gk_def}. Thus we have established \eqref{eq:Ak2_E}.

We now approach $(G^n_k)^2$.
To keep our notation manageable, during the remainder of this proof we will write $\1_\v{c}=\1_{(\v{c}_{k}\in\v{c}_n^\downarrow)}$ and $\1_\v{s}=\1_{(\v{s}_{k}\in\v{c}_n^\downarrow)}$. We define also $\1_{!\v{c}}=1-\1_{\v{c}}$ and $\1_{!\v{s}}=1-\1_{\v{s}}$, and also $\1_{\v{c}\cup\v{s}}=\1_{(\v{c}_{k}\in \v{c}_n^\downarrow\text{ or }\v{s}_{k}\in \v{s}_n^\downarrow)}$.
From \eqref{eq:Gm_recursion_pre}, for $k<n$ we have
\begin{align*}
G^n_k&=(G^n_{k+1}-A^n_k+B^n_k)\l(\1_{\v{c}}+\1_{!\v{c}}\1_{\v{s}}+\1_{!\v{c}}\1_{!\v{s}}\r)
\end{align*}
because the final bracket sums to $1$. Note that $B^n_k=\1_{\v{c}}R_{k}$.
Note also that if $\1_{!\v{c}}\1_{!\v{s}}=1$ then $A^n_k=0$.
Thus,
\begin{align*}
G^n_k
&= \l(G^n_{k+1}-A^n_k+R_{k}\r)\1_{\v{c}}+\l(G^n_{k+1}-A^n_k\r)\1_{!\v{c}}\1_{\v{s}}+G^n_{k+1}\1_{!\v{c}}\1_{!\v{s}}.
\end{align*}
Squaring both sides,
\begin{align}
(G^n_k)^2
&= \l(G^n_{k+1}-A^n_k+R_{k}\r)^2\1_{\v{c}}+\l(G^n_{k+1}-A_k^n\r)^2\1_{!\v{c}}\1_{\v{s}}+(G^n_{k+1})^2\1_{!\v{c}}\1_{!\v{s}}\notag\\
&=(G^n_{k+1})^2+2G^n_{k+1}\l[(R_{k}-A^n_k)\1_\v{c}-A^n_k\1_{!\v{c}}\1_\v{s}\r]+(R_{k}-A^n_k)^2\1_\v{c}+(A^n_k)^2\1_{!\v{c}}\1_{\v{s}}\notag\\
&\leq (G^n_{k+1})^2+2G^n_{k+1}\l[R_{k}\1_{\v{c}}-A^n_k\1_{\v{c}\cup\v{s}}\r] + R_{k}^2\1_{\v{c}} + (A^n_k)^2\1_{\v{c}\cup\v{s}}\notag\\
&= (G^n_{k+1})^2+2G^n_{k+1}\l[R_{k}\1_{\v{c}}-A^n_k\r] + R_{k}^2\1_{\v{c}} + (A^n_k)^2. \label{eq:Gk2_lower_pre}
\end{align}
To deduce the third line of the above from the second, we recall that $A^n_k\geq 0$, and to deduce the final line we use also that $\1_{\v{c}\cup\v{s}}=1$ if and only if $A^n_k\neq 0$.

We now prepare to take conditional expectation of both sides of \eqref{eq:Gk2_lower_pre}, with respect to $\mathscr{G}_{k+1}$. With this goal in mind we note that
$$\E\l[\1_\v{c}\,\Big|\,\mathscr{G}_{k+1}\r]
=1-\l(1-\frac{1}{2k+2}\r)^{G^n_{k+1}}
\leq \frac{G^n_{k+1}}{2k+2}
$$
The first equality follows from the same calculation as in \eqref{eq:Gk_gain_inter_pre} and \eqref{eq:Gk_gain_inter} (but without the $R_{k}$ term present), and the inequality then follows from \eqref{eq:taylor}.
Recall that $G^n_{k+1}$ is $\mathscr{G}_{k+1}$ measurable, but that $R_{k}$ is independent of $\mathscr{G}_{k+1}$ and of $\1_{\v{c}}$. Lastly, recall that we have Lemma \ref{lem:Gk_1st} and \eqref{eq:Ak2_E} to control $\E[A^n_k\|\mathscr{G}_{k+1}]$ and $\E\l[(A^n_k)^2\|\mathscr{G}_{k+1}\r]$ respectively.
Putting all these facts together, we obtain from \eqref{eq:Gk2_lower_pre} that
\begin{align*}
&\E[(G^n_k)^2\|\mathscr{G}_{k+1}]\\
&\hspace{2pc}\leq (G^n_{k+1})^2+2G^n_{k+1}\l(\E[R]\frac{G^n_{k+1}}{2k+2}-\frac{G^n_{k+1}}{k+1}\r)+\E[R^2]\frac{G^n_{k+1}}{2k+2}+\frac{G^n_{k+1}}{k+1}+\l(\frac{G^n_{k+1}}{k+1}\r)^2\notag\\
&\hspace{2pc}=(G^n_{k+1})^2\l(1+\frac{2(\zeta-1)}{k+1}+\frac{1}{(k+1)^2}\r)+\frac{\E[R^2]+2}{2}\frac{G^n_{k+1}}{k+1}
\end{align*}
To ease our notation, and with a view to eventually applying Lemma \ref{lem:key_asym}, for the remainder of this proof we will write $\gamma_j=j^{-2}$ and $\theta=\frac{1}{2}(\E[R^2]+2)$. Thus, taking expectations, we obtain
\begin{equation}\label{eq:Gk2_iter_pre}
\E\l[(G^n_k)^2\r]\leq \E\l[(G^n_{k+1})^2\r]\l(1+\frac{2(\zeta-1)}{k+1}+\gamma_{k+1}\r)+\theta\frac{\E[G^n_{k+1}]}{k+1}
\end{equation}
Therefore, iterating \eqref{eq:Gk2_iter_pre},
\begin{align}
\E\l[(G^n_k)^2\r]
&=\E\l[(G^n_n)^2\r]\prod\limits_{j=k+1}^n\l(1+\frac{2(\zeta-1)}{j}+\gamma_j\r)+\theta\sum\limits_{j=k+1}^n\frac{\E[G^n_j]}{j}\prod\limits_{l=k+1}^{j-1}\l(1+\frac{2(\zeta-1)}{l}+\gamma_l\r)\label{eq:Gk2_iter_pre_2}
\end{align}
We now apply Lemma \ref{lem:key_asym}, but we must be careful because of the summation over $j$. 
Let $\epsilon>0$.
By Lemma \ref{lem:key_asym} there exists $K=K_\epsilon\in\N$ such that for all $j>k\geq K$ we have
$\prod_{l=k+1}^{j-1}(1+\frac{2(\zeta-1)}{l}+\gamma_l)\leq (\frac{j-1}{k+1})^{2(\zeta-1)}(1+\epsilon).$
Similarly, by Lemma \ref{lem:Gk_1st_upper}, increasing $K$ if necessary, for all $n\geq j\geq K$ we have that 
$\E[G^n_j]\leq 2\zeta(\frac{n}{j})^{\zeta-1}(1+\epsilon)$.
Putting both these facts into \eqref{eq:Gk2_iter_pre_2}, 
for $n\geq k\geq K$ we have
\begin{align}
\E\l[(G^n_k)^2\r]
&\leq \E[R^2]\l(\frac{n}{k}\r)^{2(\zeta-1)}(1+\epsilon)+2\zeta\theta\sum\limits_{j=k+1}^n \l(\frac{n}{j}\r)^{\zeta-1}\frac{1}{j}\l(\frac{j-1}{k+1}\r)^{2(\zeta-1)}(1+\epsilon)^2\notag\\
&\leq (1+\epsilon)^2\l(\E[R^2]\l(\frac{n}{k}\r)^{2(\zeta-1)}+2\zeta\theta\sum\limits_{j=k+1}^n \l(\frac{n}{j}\r)^{\zeta-1}\frac{1}{j}\l(\frac{j}{k}\r)^{2(\zeta-1)}\r)\notag\\
&=(1+\epsilon)^2\l(\E[R^2]\l(\frac{n}{k}\r)^{2(\zeta-1)}+2\zeta\theta\l(\frac{n}{k}\r)^{\zeta-1}\frac{1}{k^{\zeta-1}}\sum\limits_{j=k+1}^n j^{\zeta-2}\r).\label{eq:Gk2_iter}
\end{align}
We have $\zeta>1$, so
$$\sum_{j=k+1}^n j^{\zeta-2}
\leq \int_{k}^{n+1} j^{\zeta-2}\,dj
\leq \frac{1}{\zeta-1}(n+1)^{\zeta-1}.
$$
Again increasing $K$ if necessary, we may assume that for $n\geq K$ we have $\sum_{j=k+1}^n j^{\zeta-2}\leq \frac{1}{\zeta-1}n^{\zeta-1}(1+\epsilon)$.
Thus, continuing from \eqref{eq:Gk2_iter} we obtain that for $n\geq k \geq K$,
\begin{align}
\E\l[(G^n_k)^2\r]
&\leq (1+\epsilon)^3\l( \E[R^2]\l(\frac{n}{k}\r)^{2(\zeta-1)}+2\frac{\zeta\theta}{\zeta-1}\l(\frac{n}{k}\r)^{\zeta-1}\l(\frac{n}{k}\r)^{\zeta-1}\r)\label{eq:Gk2_iter_post}\\
&=(1+\epsilon)^3\l(\E[R^2]+2\frac{\zeta\theta}{\zeta-1}\r)\l(\frac{n}{k}\r)^{2(\zeta-1)}.\notag\\
&=(1+\epsilon)^3\xi \l(\frac{n}{k}\r)^{2(\zeta-1)}\notag
\end{align}
In the above we take $\xi=\E[R^2]+2\frac{\zeta\theta}{\zeta-1}$.
Since $\E[R^2]\geq 4\zeta^2$ we have $\theta\geq 2\zeta^2+1$ and thus $\xi> 8\zeta^2+2$.
The stated result follows, because $\epsilon>0$ was arbitrary.
\end{proof}

As we can see from \eqref{eq:Gk2_iter_post}, the second term on the right hand side of \eqref{eq:Gk2_iter_pre} turns out, after iteration, to be of the same order as the first. Roughly speaking, the first term of \eqref{eq:Gk2_iter_post} corresponds to branching, and the second to (an over-estimate of) coalescing.
We will see the same pattern in the calculations following \eqref{eq:Gk_lower_control}, below.

We will now set $k$ to depend on $n$, such that $k_n\sim Cn^\beta$ with a suitable $C\in(0,\infty)$. This corresponds to the furthest backwards in time (from $n$) that the iterative methods used in this section are capable of seeing. In the first part of next lemma we can see that even reducing the value of $C$ towards $0$ results in the lower bound for $\frac{1}{k}\E[G^n_k]$ drifting away from the upper bound.
This root cause is as follows: in order to remain tractable the proof of Lemma \ref{lem:Gk2_upper} has used a slight underestimate of the coalescence part, in \eqref{eq:Gk2_lower_pre}. Once coalescences become non-negligible, which occurs around $k\approx n^\beta$ when $\frac{1}{k}\E[G^n_k]$ becomes non-trivial, that estimate breaks down. 

\begin{lemma}\label{lem:branching_end_control}
Let $C\in(0,\infty)$. Suppose that $k=k_n$ is such that $k\leq n$ and $k\sim Cn^\beta$.  
Then as $n\to\infty$ we have
\begin{enumerate}
\item $\frac{2\zeta}{C^\zeta}\l(1-\frac{\xi}{8\zeta C^{\zeta}}\r)\,\lesssim\, \frac{1}{k}\E[G^n_k] \,\lesssim\, \frac{2\zeta}{C^{\zeta}}$,
\item $\frac{1}{k^2}\E\l[(G^n_k)^2\r]\lesssim \frac{\xi}{C^{2\zeta}}.$
\end{enumerate}
\end{lemma}

\begin{proof}
Let us first prove the two asymptotic upper bounds.
Recall that $\beta\in(0,1)$ was defined in \eqref{eq:zeta_beta_def} and
note that $(1-\beta)\zeta=1$.
Thus, since $k\sim Cn^\beta$, we have $\l(\frac{n}{k}\r)^{\zeta-1}\sim\l(\frac{n}{Cn^\beta}\r)^{\zeta-1}=\frac{1}{C^{\zeta-1}}n^\beta\sim\frac{1}{C^\zeta}k$. The first upper bound now follows from Lemma \ref{lem:Gk_1st_upper} and the second upper bound follows from Lemma \ref{lem:Gk2_upper}.
It remains to prove the lower bound from the first claim of the lemma.
We have
\begin{align}
\E[G^n_k\|\mathscr{G}_{k+1}]
&\geq G^n_{k+1}-\frac{G^n_{k+1}}{k+1}+2\zeta\l(\frac{G^n_{k+1}}{2k+2}-\binom{G^n_{k+1}}{2}\frac{1}{4(k+1)^2}\r)\notag\\
&=G^n_{k+1}\l(1+\frac{\zeta-1}{k+1}\r)-\zeta\binom{G^n_{k+1}}{2}\frac{1}{2(k+1)^2}\notag\\
&\geq G^n_{k+1}\l(1+\frac{\zeta-1}{k+1}\r)-\frac{\zeta}{4}\l(\frac{G^n_{k+1}}{k+1}\r)^2.\label{eq:Gk_lower_control}
\end{align}
Here, the first line follows from Lemma \ref{lem:Gk_1st} and the right hand side of \eqref{eq:taylor},
and the second and third lines are elementary computations.
Iterating \eqref{eq:Gk_lower_control}, we obtain
\begin{align}
\E[G^n_k]
&\geq \E[G^n_n]\prod\limits_{j=k+1}^{n}\l(1+\frac{\zeta-1}{j}\r)-\frac{\zeta}{4}\sum\limits_{j=k+1}^n\frac{\E[(G^n_j)^2]}{j^2}\prod\limits_{l=k+1}^{j-1}\l(1+\frac{\zeta-1}{l}\r).\label{eq:Gk_lower_control_2}
\end{align}
We now seek to apply Lemma \ref{lem:key_asym}, but again we must take care to handle the summation over $j$.
Let $\epsilon>0$.
By Lemma \ref{lem:key_asym} there exists $K=K_\epsilon\in\N$ such that for all $j>k\geq K$ we have
$$
\l(\frac{j}{k}\r)^{\zeta-1}(1-\epsilon)\leq \prod_{l=k+1}^{j-1}\l(1+\frac{\zeta-1}{l}\r)\leq \l(\frac{j}{k}\r)^{\zeta-1}(1+\epsilon).$$
Similarly, by Lemma \ref{lem:Gk2_upper}, increasing $K$ if necessary, we may assume that for all $n\geq j\geq K$ we have
$\E[(G^n_j)^2]\leq (\frac{n}{j})^{2(\zeta-1)}(1+\epsilon)$.
We thus obtain from \eqref{eq:Gk_lower_control_2} that for $n\geq j>k\geq K$,
\begin{align}
\E[G^n_k]
&\geq 2\zeta\l(\frac{n}{k}\r)^{\zeta-1}(1-\epsilon)-\frac{\zeta\xi}{4}\sum\limits_{j=k+1}^n\l(\frac{n}{j}\r)^{2(\zeta-1)}\l(\frac{j}{k}\r)^{\zeta-1}(1+\epsilon)^2\notag\\
&=(1-\epsilon)(2\zeta)\l(\frac{n}{k}\r)^{\zeta-1}-(1+\epsilon)^2\frac{\zeta\xi}{4}\l(\frac{n}{k}\r)^{\zeta-1}n^{\zeta-1}\sum\limits_{j=k+1}^n j^{-\zeta-1}.\label{eq:Gk_lower_control_post}
\end{align}
We have $\zeta>1$, so
$\sum_{j=k+1}^n j^{-\zeta-1}\leq \int_{k}^{n} j^{-\zeta-1}\,dj\leq \frac{1}{\zeta}k^{-\zeta}.$
Thus, from \eqref{eq:Gk_lower_control_post} we obtain
\begin{align*}
\E[G^n_k]
&\geq (1-\epsilon)(2\zeta)\l(\frac{n}{k}\r)^{\zeta-1}-(1+\epsilon)^2\frac{\xi}{4}\l(\frac{n}{k}\r)^{\zeta-1}n^{\zeta-1}k^{-\zeta}\\
&=\l(\frac{n}{k}\r)^{\zeta-1}\l((1-\epsilon)(2\zeta)-(1+\epsilon)^2\frac{\xi}{4}\l(\frac{n}{k}\r)^{\zeta-1}\frac{1}{k}\r).
\end{align*}
Recall that $k\sim Cn^\beta$ and $(1-\beta)\zeta=1$. Therefore, increasing $k$ again if necessary, we may assume that for all $n\geq k\geq K$ we have $\l(\frac{n}{k}\r)^{\zeta-1}\frac{1}{k}\geq \frac{1}{C^\zeta}(1-\epsilon)$.
Thus for $n\geq k\geq K$ we have
\begin{align*}
\frac{1}{k}\E[G^n_k]
&\geq 
\frac{1}{C^\zeta}(1-\epsilon)\l((1-\epsilon)(2\zeta)-(1+\epsilon)^2\frac{\xi}{4}\frac{1}{C^\zeta}(1-\epsilon)\r)\\
&\geq
(1-\epsilon)^2\frac{2\zeta}{C^\zeta}\l(1-\frac{\xi}{8\zeta C^\zeta}\r)-3\epsilon\frac{\xi}{4C^\zeta}.
\end{align*}
The asymptotic lower bound claimed in the first part of the lemma now follows because $\epsilon>0$ was arbitrary. This completes the proof.
\end{proof}


\begin{proof}[Of Proposition \ref{prop:branching_phase_pickup}.]
Recall the definition on $N^n_{i,i'}$ in \eqref{eq:Nkip_def}.
Let $C>0$. We have
\begin{align}
\E\l[N^n_{Cn^\beta,n}\r]
&=\sum\limits_{k=Cn^\beta}^n\E\l[\sum\limits_{j=k+1}^n\1_{\l(\v{c}_j\in\v{c}_n^\downarrow\r)}\sum\limits_{l=1}^{R_j}\1_{\l(\v{p}_{j,l}\in\{\v{c}_{k},\v{s}_{k}\}\r)}\1_{\l(\v{p}_{j,l}=\v{s}_k\r)}\r]\notag\\
&=\sum\limits_{k=Cn^\beta}^n\E\l[\sum\limits_{j=k+1}^n\1_{\l(\v{c}_j\in\v{c}_n^\downarrow\r)}\sum\limits_{l=1}^{R_j}\1_{\l(\v{p}_{j,l}\in\{\v{c}_{k},\v{s}_{k}\}\r)}\r]\frac{1}{2}\notag\\
&=\frac{1}{2}\sum\limits_{k=Cn^\beta}^n\E\l[A^n_k\r].\label{eq:Nn_pre}
\end{align}
In the above, the first line follows from \eqref{eq:Nk_def}, \eqref{eq:Nkip_def} and from noting that $\v{s}_k\in\{\v{s}_k,\v{c}_k\}$.
To deduce the second line we use $(\dagger)$ which implies that, for each $j$, $\P[\v{p}_{j,l}=\v{s}_k\|\v{p}_{j,l}\in\{\v{s}_k,\v{c}_k\}, \v{c}_j\in\v{c}_n^\downarrow]=\frac12$.
The third line then follows by \eqref{eq:Gk_loss}.

Combining Lemmas \ref{lem:Gk_1st} and \ref{lem:Gk_1st_upper} we obtain $\E[A^n_k]=\frac{\E[G^n_{k+1}]}{k+1}\lesssim 2\zeta \frac{1}{k}\l(\frac{n}{k}\r)^{\zeta-1}$. 
Thus, for any $\epsilon>0$ there exists $K$ such that for all $n\geq k\geq K$, 
$\E[A^n_k]\leq 2\zeta \frac{1}{k}\l(\frac{n}{k}\r)^{\zeta-1}(1+\epsilon)$. 
Combing this with \eqref{eq:Nn_pre} obtains
\begin{align*}
\E\l[N^n_{Cn^\beta,n}\r]
&\leq (1+\epsilon)\zeta n^{\zeta-1}\sum\limits_{k=Cn^\beta}^n k^{-\zeta}\\
&\leq (1+\epsilon)\zeta n^{\zeta-1} \frac{(Cn^\beta)^{1-\zeta}}{\zeta-1}\\
&=(1+\epsilon)\frac{n^\beta}{\beta C^{\zeta-1}}.
\end{align*}
To deduce the final line of the above we use $\beta\zeta=\zeta-1$.
The proposition follows because $\epsilon>0$ was arbitrary.
\end{proof}

\subsection{The branching-coalescing phase of the genealogy}
\label{sec:ancestry_stablization}

We now turn our attention to look further backwards into the genealogy of $\v{c}^\downarrow_n$, in particular at the full range of times of order $n^\beta$.
It is during this window of time that coalescences become frequent.
With this in mind, we will now need to count potential parents of ancestors of $\v{c}^\downarrow_n$ \textit{without} multiplicity. In fact, it will also become useful to record their identities.
To this end, for convenience we reproduce \eqref{eq:H_def} here:
\begin{align}
\mc{H}^n_k&=\{\v{u}\in\mc{U}_{k}\-\v{u}\in\mc{P}_j\text{ for some }\v{c}_j\in\v{c}^\downarrow_n\text{ with }j\geq k+1\},\notag\\
H^n_{k}&=|\mc{H}^n_k|.\notag
\end{align}
Note that $\mc{H}^n_k$ is the set of balls that were born (non-strictly) before time $k$,
and were a potential parent of some $\v{c}_j\in\v{c}^\downarrow_n$ where $j>k$.
Thus, comparing to \eqref{eq:Gk_def}, $H^n_k$ is `$G^n_{k+1}$ counted without multiplicity'.
The apparently incongruity between $k$ and $k+1$ will not bother us, because we will shortly shift our emphasis entirely from $G^n_k$ to $H^n_k$. Moreover, in this section is it advantageous that $H^n_k$ relates to $\mc{U}_k$ and not to $\mc{U}_{k-1}$.
Let us now upgrade $(\dagger)$ and Lemma \ref{lem:branching_end_control} to handle $H^n_k$.
\begin{itemize}
\item[$(\dagger\dagger)$] The conditional distribution of $\mc{H}^n_k$ given $H^n_k$ is uniform on the set of subsets of $\mc{U}_{k}$ that have size $H^n_k$.
\end{itemize}
\begin{proof}[Of $(\dagger\dagger)$.] Recall the definition of $G^n_{k+1}$ from \eqref{eq:Gk_def}: it counts the number of times a parent of some $\v{c}^\downarrow_j$ (with $k+1\leq j\leq n$) was an element of $\mc{U}_{k}$. By $(\dagger)$, each such parent is a uniformly sampled element of $\mc{U}_{k}$, independently of all else.
\end{proof}

For the remainder of Section \ref{sec:ancestry_stablization}, we fix a pair of constants $c,C$ such that $0<c<C<\infty$. 
It is understood that they will be chosen dependent upon the common distribution of the $R_k$.
We assume, without loss of generality, that both $cn^\beta$ and $Cn^\beta$ are integer c.f.~Remark \ref{rem:Cinteger}.

\begin{lemma}
\label{lem:br_end_H}
Suppose that $k=k_n\sim Cn^\beta$. Then, as $n\to\infty$ we have
\begin{enumerate}
\item $\frac{2\zeta}{C^\zeta}\l(1-\frac{\xi}{8\zeta C^{\zeta}}\r)\,\lesssim\, \frac{1}{k}\E[H^n_k] \,\lesssim\, \frac{2\zeta}{C^{\zeta}}$,
\item $\frac{1}{k^2}\E\l[(H^n_k)^2\r]\lesssim \frac{\xi}{C^{2\zeta}}.$
\end{enumerate}
\end{lemma}
\begin{proof}
It is immediate that $H^n_k\leq G^n_{k+1}$.
Thus both asymptotic upper bounds follow from their counterparts in Lemma \ref{lem:branching_end_control}.
For the lower bound, by symmetry
we have 
$\E[H^n_k\|G^n_{k+1}]=2k\P[\v{b}\in\mc{H}^n_k\|G^n_{k+1}],$ 
where $\v{b}$ is any fixed ball in $\in\mc{U}_k$. 
By $(\dagger)$ we have $\P[\v{b}\notin\mc{H}^n_k\|G^n_{k+1}]=\big(1-\frac{1}{2k}\big)^{G^n_{k+1}}$, 
so $\E[H^n_k\|G^n_{k+1}]=2k\big(1-(1-\frac{1}{2k})^{G^n_{k+1}})$.
Using \eqref{eq:taylor} and taking expectations, we thus obtain that
$\E[G^n_{k+1}]-\frac{1}{2}\frac{1}{2k}\E[(G^n_{k+1})^2]\leq \E[H^n_k].$
Then, using Lemma \ref{lem:branching_end_control} again we have
$\frac{1}{k}\E[H^n_k]\gtrsim \frac{2\zeta}{C^\zeta}-\frac14\frac{\xi}{C^{2\zeta}}$ 
as required. 
\end{proof}

It will be helpful to work with proportions of cue balls rather than with their absolute number. We set
\begin{align}
Z^{n}_j=\frac{1}{2(Cn^\beta-j)+2},\quad\quad\quad\quad Y^{n}_j=Z^{n}_j H^n_{Cn^\beta-j}\label{eq:Ynj_def}
\end{align}
defined for $j=0,1,\ldots,(C-c)n^\beta$. In words, $Z^n_j$ is one over the the number of balls born (non-strictly) before time $Cn^\beta-j$, and $Y^{n}_j$ is the proportion of such balls that constitute $\mc{H}^n_{Cn^\beta-j}$.
The indexing in \eqref{eq:Ynj_def} is in preparation for finding a fluid limit, as $n\to\infty$, of the process $k\mapsto \frac{1}{k}H^n_{k}$ considered \textit{backwards in time} -- that is as $j$ increases and $k$ decreases, during the critical window $k\in[cn^\beta, Cn^\beta]$.
This is somewhat tricky because the process $H^n_k$ is non-Markov, with respect to its generated filtration, and also time-inhomogeneous.
We will see that these technical difficulties may be overcome by taking the limit of $(Y^n_j, Z^n_j)$ under a suitable rescaling of time.

We are now ready to state the major result of this section, Proposition \ref{prop:stablization_bounds}, which will come as a consequence of the aforementioned fluid limit.
We write $Y^n_u=Y^n_j$ for any $u\in[j,j+1)$.

\begin{prop}\label{prop:stablization_bounds}
If $0<c<C<\infty$ and $C>\xi^{1/\zeta}$ then
\begin{equation}\label{eq:stablization_bounds}
\P\l[\text{for all }s\in[0,1],\;Y^n_{s(C-c)n^\beta}\geq \frac{\zeta}{\zeta+e^{c/C}(C-s(C-c))^{2\zeta}}\r]\gtrsim \frac{\zeta^2}{4\xi}.
\end{equation}
\end{prop}

Let us comment a little further on the strategy we adopt to prove Proposition \ref{prop:stablization_bounds}, and outline where the formula \eqref{eq:stablization_bounds} comes from. 
We look to obtain a fluid limit for the $[0,1]^2$ valued process $(Y^n_j, Z^n_j)$, during time $j=0,1,\ldots,(C-c)n^\beta$.
We will use the framework of weak convergence.
To this end, we parametrize time using $s\in[0,1]$, resulting in times $s(C-c)n^\beta$, but we will see that it is also helpful to make the substitution $t=\log(\frac{C}{C-s(C-c)})$ after which, loosely speaking, the limit of $Y^n_j$ will turn out to be the ordinary differential equation
\begin{equation}\label{eq:stablization_ODE}
\frac{d y(t)}{dt}=2\zeta\, y(t)(1-y(t))
\end{equation}
run for time $t\in[0,\log(C/c)]$, starting from the initial condition $y(0)\approx Y^n_0$.
Equation \eqref{eq:stablization_ODE} is well known -- it is logistic growth at rate $2\zeta$.
The precise formulation of \eqref{eq:stablization_bounds} comes from the explicit solution to \eqref{eq:stablization_ODE}, which 
is 
$y(t)=\frac{A}{A-(A-1)e^{-2\zeta t}}$, where $y(0)=A$.
The limit of $Z^n_j$ will be zero; its presence is solely because $(Y^n_j, Z^n_j)$ is a time-homogeneous Markov process, whereas $(Y^n_j)$ alone is not.

The ODE \eqref{eq:stablization_ODE} has a stable fixed point at $1$ and an unstable fixed point at $0$.
Our initial condition $y(0)$ is positive, resulting in attraction towards $1$ as $t$ increases.
However, the value of $y(0)\approx \frac{2\zeta}{C^\zeta}$ tends to zero as $C\to\infty$, 
as a consequence of Lemma \ref{lem:br_end_H}.
We need to keep enough freedom to choose a large value for $C$ (as we did in Section \ref{sec:ext_cond_proof}).
Heuristically, as $C\to\infty$ we observe \eqref{eq:stablization_ODE} started with a vanishing initial displacement, of order $C^{-\zeta}$, away from its unstable fixed point at $y=0$.
It is not a priori clear if the time interval $t\in[0,\log(C/c)]$ gives long enough to actually escape from $0$; fortunately, we will see that it does. Proposition \ref{prop:stablization_bounds} comes from knowing that $Y^n_j$ behaves similarly to $y(t)$, for large $n$.

With Proposition \ref{prop:stablization_bounds} in hand, the proof of Proposition \ref{prop:ancestry_stablization} is straightforward, so we will give it now; then we will turn our attention to proving Proposition \ref{prop:stablization_bounds}.

\begin{proof}[Of Proposition \ref{prop:ancestry_stablization}, subject to Proposition \ref{prop:stablization_bounds}.]
We have $c\in(0,2)$ and $C>\max(c,\xi^{1/\zeta})$. Thus $\frac{C-2}{C-c}\in(0,1)$.
For $s\in[\frac{C-2}{C-c},1]$ we have
$$\frac{\zeta}{\zeta+e^{c/C}(C-s(C-c))^{2\zeta}}\geq \frac{\zeta}{\zeta+e^{c/C}4^\zeta}.$$
Thus, by Proposition \ref{prop:stablization_bounds}, with probability at least $\frac{\zeta^2}{4\xi}$, we have that
$Y^n_{s(C-c)n^\beta}\geq \frac{\zeta}{\zeta+e^{c/C}4^\zeta}$ for all $s\in[\frac{C-2}{C-c},1]$.
Recalling that $Y^n_j=H^n_{Cn^\beta-j}$, and noting that $s=\frac{C-2}{C-c}$ corresponds to $j=2n^\beta$ whilst $s=1$ corresponds to $j=cn^\beta$,
this gives
$H^n_j\geq \frac{\zeta}{\zeta+e^{c/C}4^\zeta}$ for all $j=cn^\beta,\ldots,2n^\beta$.
The result follows.
\end{proof}


The remainder of this section will focus on proof of Propositon \ref{prop:stablization_bounds}. As we have mentioned, $j\mapsto Y^{n}_j$ is not time-homogeneous, which is due to its dependence on $Z^n_j$. 
However, $X^n_j=(Y^n_j,Z^n_j)$ \textit{is} time-homogeneous, as we will now show.
We require some notation. We define (the function $\cdot'$ as)
\begin{equation*}
z'=\l(\frac{1}{z}-2\r)^{-1}
\end{equation*}
and note that $Z^n_{j+1}=(Z^n_j)'$. 
Let $\mc{B}_{r,a,b}$ be an independent random variable defined as follows.
Take $a\in\N$ boxes, $b\in\N$ of which are marked, and distribute $r$ balls (uniformly at random, with replacement) into these $a$ boxes; $\mc{B}_{r,a,b}$ is the number of newly occupied boxes that are not marked.

\begin{lemma}
\label{lem:Xnj_markov}
Fix $n\in\N$. The process $j\mapsto X^n_j$ is a time-homogeneous $[0,1]^2$ valued Markov process, in discrete time $j=0,1,\ldots,(C-c)n^\beta$.
The law of its one-step transition, given that $X^n_j=(y,z)$, is that of 
\begin{equation}\label{eq:Yj_pregen}
(y,z)\mapsto
\l(y\tfrac{z'}{z}-z'I_0+z'I_1\l(-1+\mc{B}_{R,1/z',y/z}\r),\;z'\r)\\
\end{equation}
where 
$(I_0,I_1)$ is a pair of Bernoulli random variables with distribution given by
$\P[I_0=1]=\P[I_1=1]=y$ and $\P[I_0=1,I_1=1]=\frac{y^2-yz}{1-z}$,
and $R$ is an independent sample of the common distribution of the $(R_j)$.
\end{lemma}
\begin{proof}
Consider the one-step transition of $X^n_j$ to $X^n_{j+1}$
where $j=Cn^\beta-k$.
Suppose that $X^n_j$ is at location $(y,z)\in[0,1]^2$, so that
$$y=\frac{H^n_{k}}{|\mc{U}_{k}|},\quad\quad z=\frac{1}{|\mc{U}_{k}|}.$$
In the $z$ coordinate of $X^n_\cdot$, we see the deterministic movement $Z^n_{j}\mapsto Z^n_{j+1}$, which is simply $z\mapsto z'$.
To understand the random movement within the $y$ coordinate, we must first consider the transition $\mc{H}^n_k\mapsto\mc{H}^n_{k-1}$.
To construct $\mc{H}^n_{k-1}$ from $\mc{H}^n_{k}$ we must do both of:
\begin{enumerate}
\item check if $\v{s}_{k}\in \mc{H}^n_{k}$; if it is then we must remove $\v{s}_{k}$,
\item check if $\v{c}_{k}\in \mc{H}^n_{k}$; if it is then we must remove $\v{c}_{k}$ and add in any potential parents of $\v{c}_{k}$ that were not already there.
\end{enumerate}
By $(\dagger\dagger)$, the events $\v{s}_{k}\in \mc{H}^n_{k}$ and $\v{c}_{k}\in \mc{H}^n_{k}$ are not independent.
However, by $(\dagger\dagger)$ and using exchangeability we do have
\begin{align}
\P\l[\v{s}_{k}\in \mc{H}^n_{k}| H^n_{k}\r]=\P\l[\v{c}_{k}\in \mc{H}^n_{k}| H^n_{k}\r]
=\frac{H^n_{k}}{|\mc{U}_{k}|}
&=y\label{eq:IcIs1}\\
\P\l[\v{s}_{k}\in \mc{H}^n_{k}, \v{c}_{k}\in \mc{H}^n_{k}| H^n_{k}\r]
=\frac{H^n_{k}}{|\mc{U}_{k}|}\frac{H^n_{k}-1}{|\mc{U}_{k}|-1}
&=\frac{y^2-yz}{1-z}.\label{eq:IcIs2}
\end{align}
Let $(I_0,I_1)$ denote a pair of correlated Bernoulli random variables, taking values in $\{0,1\}^2$, with the distribution of $\l(\1_{\l(\v{s}_{k}\in \mc{H}^n_{k}\r)},\1_{\l(\v{c}_{k}\in \mc{H}^n_{k}\r)}\r)$ given $H^n_{k}$, that is
$$\P[I_0=1]=\P[I_1=1]=y,\quad\quad\P[I_\v{c}=1,I_\v{c}=1]=\frac{y^2-yz}{1-z}.$$
We claim that, on the event $\{\v{c}_k\in\v{c}^\downarrow_n\}$, the number of potential parents of $\v{c}_k$ that are in $\mc{H}^n_{k}\sc\mc{H}^n_{k-1}$ has the same distribution as an independent copy of
$\mc{B}_{R_k,1/z', y/z}$.
To see this, note that $\mc{H}^n_{k-1}=y/z$ and $|\mc{U}_{k-1}|=1/z'$. 
The claim now follows from $(\dagger\dagger)$ and the definition of $\mc{B}_{r,a,b}$.

Note that we took $j=Cn^\beta-k$, so the transition $Y^n_j\mapsto Y^n_{j+1}$ is precisely the transition of
$Z^n_jH^n_{k}\mapsto Z^n_{j+1}H^n_{k-1}$.
Putting this together with all the above, we obtain that, conditionally given $X^n_j=(y,z)$, the transition $X^n_{j}\mapsto X^n_{j+1}$ has the same law as
\begin{equation*}
(y,z)\mapsto
\l(y\tfrac{z'}{z}-z'I_0+z'I_1\l(-1+\mc{B}_{R_k,1/z',y/z}\r),\;z'\r),
\end{equation*}
as required.
Conditionally given the value of $(y,z)$, the law of this transition does not depend on time (i.e.~on $j$ or $k$). Thus $(X^n_j)$ is a time-homogeneous process.
Note that $I_0,I_1$ are all measurable with respect to $\sigma(X^n)$, and that the distribution of $(I_0,I_1)$ depends only on the current state $X^n_j=(y,z)$. 

Hence, $(X^n_j)$ is Markov with respect to the filtration generated by its own motion, the i.i.d.~sequence $(R_{Cn^\beta},R_{Cn^\beta-1},\ldots,R_1)$, and the i.i.d.~random functions $\{\mc{B}_{\cdot,i,i'}\-i'\leq i\}$. (Strictly, we must include a new independent copy of the latter on each time-step.)
\end{proof}

In order to take a fluid limit, it will be convenient to work in continuous time.
Essentially this just means embedding $X^n_j$ as a the jump chain of a continuous time Markov process,
but we also need to handle one more technicality:
our fluid limit only exists around the time of the critical window $[cn^\beta,Cn^\beta]$.
Consequently, once the critical window of time has passed, it is convenient to make the pre-limiting processes behave exactly like their limit. Note that the critical window of time ends when $Z^n_j$ rises above $\frac{1}{2cn^\beta+2}$.
So, we define a continuous time Markov process $\mc{X}^{n}_t$, taking values in $[0,1]^2$, as follows. We write $\mc{X}^{n}_t=(\mc{Y}^n_t,\mc{Z}^n_t)$.
\begin{itemize}
\item Whilst $\mc{Z}^n_t<\frac{1}{2cn^\beta+2}$, the process $\mc{X}^{n}_t$ evolves with its jump chain having the same dynamics as $(Y^{n}_j,Z^n_j)$. The holding time of step $j\mapsto j+1$ is exponential with mean $Z^n_j$.
\item If $\mc{Z}^n_t\geq \frac{1}{2cn^\beta+2}$, then $\mc{Z}^n_t$ remains constant and $\mc{Y}^n_t$ evolves deterministically according to \eqref{eq:stablization_ODE}.
\end{itemize}
We have not yet chosen an initial state for $\mc{X}^n_0$ (and we will not do so, yet). However, the following lemma is immediate:

\begin{lemma}\label{lem:XXcouple}
Suppose that $X^n_0=\mc{X}^n_0$. Let $(E^n_j)$ be a sequence of independent random variables with distribution $E_j\sim Exp(1/z^n_j)$, and define
$T^n_j=\sum_{l=1}^j E^n_j.$
Then, there exists a coupling between $(E^n_j)$, $(X^n_j)$ and $(\mc{X}^n_t)$, such that $X^n_j=\mc{X}^n_{T^n_j}$ for all $j$.
\end{lemma}

Let $\mc{X}_t=(\mc{Y}_t,\mc{Z}_t)$ be the time-homogeneous Markov process taking values in $[0,1]^2$ in which the first coordinate evolves according to \eqref{eq:stablization_ODE}, and the second coordinate stays constant.
It is easily seen that both $\mc{X}^n_t$ and $\mc{X}_t$ are time-homogeneous strongly Markov processes.
We now begin a sequence of lemmas which will lead us to the proof of Proposition \ref{prop:stablization_bounds}.

We will establish in Lemmas \ref{lem:X_tightness} and \ref{lem:Xt_conv} that $\mc{X}^n_t$ converges weakly to $\mc{X}_t$, from which
it follows immediately that $\mc{Y}^n_t$ converges weakly to $\mc{Y}_t$.
We must then work back to deduce a corresponding result about $Y^n_j$ and the solution to \eqref{eq:stablization_ODE}.
Let us now outline the strategy for doing so.

Let $\mathscr{D}(E)$ denote the Skorohod space of {\cadlag} paths mapping $[0,\infty)\to E$.
We know exactly how $\mc{Y}_t$ behaves; it follows the ODE \eqref{eq:stablization_ODE} which has the explicit solution \eqref{eq:ytA}.
The content of Proposition \ref{prop:stablization_bounds} is essentially that $Y^n_{s(C-c)n^\beta}$, where $t(s)=\log(\frac{C}{C-s(C-c)})$, behaves in approximately the same way as $\mc{Y}_t$ (and this leads to the lower bound).
In order to deduce this from weak convergence we need to use the Portmanteau theorem, which requires that a suitable subset of the Skorohod space $\mathscr{D}([0,1]^2)$ is open; this appears as Lemma \ref{lem:portmanteau_prep}.
We also need to control the time change $t(s)$ and handle the fact that $\mc{Y}_t$ has continuous time but $Y^n_j$ has discrete time; this is Lemma \ref{lem:ts}.
Lastly, our information concerning the initial condition $Y^n_0=\mc{Y}^n_0$ is rather weak. Lemma \ref{lem:br_end_H} provides only asymptotic bounds on its expectation. In the final step of the proof of Proposition \ref{prop:stablization_bounds}, which comes at the end of this section, we employ an artificial trick to regain some control over this initial condition.

\begin{lemma}\label{lem:X_tightness}
The sequence of processes $(\mc{X}^n)$ is tight in $\mathscr{D}([0,1]^2)$.
\end{lemma}
\begin{proof}
By Corollary 3.9.1 of \cite{EthierKurtz1986}, it suffices to check that $(f(\mc{Y}^n_\cdot,\mc{X}^n_\cdot))$ is tight in $\mathscr{D}(\R)$, for every Lipschitz $f:[0,1]^2\to\R$. Let us fix such an $f$ and write $\mc{W}^n_t=f(\mc{Y}^n_t,\mc{Z}^n_t)$.
We will use a standard criterion of \cite{Aldous1978} to check tightness of $(\mc{W}^n)$. To this end, for each $n\in\N$ let $\tau_n$ be a stopping time, with respect to the generated filtration of $\mc{W}^n$. Note that since $\sigma(\mc{W}^n_t)\sw\sigma(\mc{X}^n_t)$, it follows immediately that $\tau_n$ is also a stopping time with respect to the filtration generated by $\mc{X}^n$. We use the latter filtration for the remainder of this proof.
The criterion of \cite{Aldous1978} requires us to show that for all $\epsilon>0$ there exist $\theta>0$ and $N\in\N$ such that for all $n\geq N$ and $s\in(0,\theta)$,
\begin{equation}\label{eq:tightness}
\P\l[|\mc{W}^n_{\tau_n+s}-\mc{W}^n_{\tau_n}|\geq \epsilon\r]\leq\epsilon.
\end{equation}
We will spend the remainder of the proof establishing that this equation holds.
Let $||f||_{Lip}$ denote the Lipshitz constant of $f$ with respect to the $L^\infty$ norm on $[0,1]^2$.
Thus,
\begin{align}\label{eq:tightness_YZ}
|\mc{W}^n_{\tau_n+s}-\mc{W}^n_{\tau_n}|\leq ||f||_{Lip}\l(|\mc{Z}^n_{\tau_n+s}-\mc{Z}^n_{\tau_n}|+|\mc{Y}^n_{\tau_n+s}-\mc{Y}^n_{\tau_n}|\r).
\end{align}
For $t$ such that $t\geq\inf\{u\-\mc{Z}^n_u\geq\frac{1}{2cn^\beta}\}$, $\mc{Z}^n_t$ remains constant and $\mc{Y}^n_t\in[0,1]$ evolves smoothly and deterministically according to the ODE \eqref{eq:stablization_ODE}. Thus, using \eqref{eq:tightness_YZ} it is easily seen that \eqref{eq:tightness} holds during this region of time.

It remains to consider $t\leq \inf\{u\-\mc{Z}^n_u\geq\frac{1}{2cn^\beta}\}$. During this region of time $\mc{X}^n$ is a jump process and the rate at which $\mc{X}^n$ jumps is bounded above by $2Cn^\beta+2$.
At these jumps, the change in magnitude of $\mc{Z}^n$ is uniformly (over the jumps) $\mc{O}(n^{-2\beta})$ and the change in $\mc{Y}^n$ is uniformly $\mc{O}(n^{-\beta})$.
Thus there exists some constant $\wt{c}\in(0,\infty)$ such that both changes in magnitude are bounded above by $\wt{c}n^{-\beta}$. Thus, as time progresses, the sum of the magnitudes of the jumps is, in both cases, stochastically bounded above by a Poisson process $\mc{V}^n_t$ that makes upwards jumps of size $\wt{c}n^{-\beta}$ at rate $2Cn^\beta+2$.  Since the jumps of $\mc{Y}^n$ and $\mc{Z}^n$ occur at the same points in time, in fact we can use a single (coupled) copy of $\mc{V}^n$ to bound them both. Thus, from, \eqref{eq:tightness_YZ}
\begin{align}\label{eq:tightness_V}
|\mc{W}^n_{\tau_n+s}-\mc{W}^n_{\tau_n}|\leq 2||f||_{Lip}\,\mc{V}^n_{s}
\end{align}

Let $T^n$ denote the time taken for the first $4\theta Cn^\beta$ jumps made by $\mc{V}^n$ (rounded upwards).
From part (iii) of Theorem 5.1 of \cite{Janson2018},
which gives tail bounds on sums of exponential random variables,
\begin{align}
\P[T^n\leq \theta]
&\leq \exp\l(-(2Cn^\beta+2)(2\theta)(\tfrac12-1-\log(\tfrac12))\r)\notag\\
&\leq \exp\l(-\theta Cn^\beta (\log 16-2)\r).\label{eq:Vn_fast}
\end{align}
Moreover, note that on the event $T^n>\theta$, $\mc{V}^n$ makes at most $4\theta Cn^\beta$ jumps during time $[0,\theta]$, so noting that $\mc{V}^n$ is an increasing process we have
\begin{align}\label{eq:Vn_slow}
T^n>\theta \quad\ra\quad \sup_{s\in[0,\theta]}\mc{V}^n_{s}\leq \frac{4\theta Cn^\beta}{\wt{c}n^\beta}\sim\frac{4\theta C}{\wt{c}}.
\end{align}

Let $\epsilon>0$. Choose $\theta=\wt{c}\epsilon/(16C||f||_{Lip})$,
which implies that the right hand side of \eqref{eq:Vn_slow} is bounded above by $\epsilon/(4||f||_{Lip})$,
and thus from \eqref{eq:tightness_V} whenever $T^n>\theta$ we have $|\mc{W}^n_{\tau_n+s}-\mc{W}^n_{\tau_n}|\leq\epsilon/2$ for all $s\in[0,\theta]$.
Choose $N=\l(\epsilon \theta C(\log 16-2)\r)^{-1/\beta}$, which implies that
for all $n\geq N$ the right hand side of \eqref{eq:Vn_fast} is bounded above by $e^{-1/\epsilon}$ and hence also by $\epsilon$ itself. Thus, by conditioning on the event $\{T^n\leq\theta\}$ we obtain that for all $s\in[0,\theta]$ and $n\geq N$,
$
\P\l[|\mc{W}^n_{\tau_n+s}-\mc{W}^n_{\tau_n}|\geq \epsilon\r]\leq \epsilon.
$
This establishes \eqref{eq:tightness} and thus completes the proof.
\end{proof}

\begin{lemma}\label{lem:Xt_conv}
Suppose $\mc{X}^n_0$ converges in law to $\mc{X}_0$. Then $\mc{X}^{n}$ converges weakly to $\mc{X}$ in $\mathscr{D}([0,1]^2)$.
\end{lemma}
\begin{proof}
First, we establish an elementary inequality relating to $\mc{B}_{r,a,b}$. 
Recall the definition of $\mc{B}_{r,a,b}$ (just above Lemma \ref{lem:Xnj_markov}) in terms of placing balls into marked and unmarked boxes.
Let $0<r\leq a-b$. We claim that 
\begin{equation}
\label{eq:Brab}
r(1-\frac{b}{a}-\frac{r}{a})\leq \E[\mc{B}_{r,a,b}]\leq r(1-\frac{b}{a}).
\end{equation}
To see \eqref{eq:Brab}, note that
we can bound $\mc{B}_{r,a,b}$ from above by counting the total number of balls placed into unmarked boxes; this is binomial with $r$ trials and success probability $\frac{a-b}{a}=1-\frac{b}{a}$.
We can bound $\mc{B}_{r,a,b}$ below by noting that at most $r$ unmarked boxes will be chosen in total, so if we place the $r$ balls in turn, then each time we place a ball the chance of it being placed into an (as yet) unoccupied unmarked box is at least $\frac{a-b-r}{a}=1-\frac{b}{a}-\frac{r}{a}$; hence $\mc{B}_{r,a,b}$ is stochastically bounded below by a binomial with $r$ trials and success probability $1-\frac{b}{a}-\frac{r}{a}$. Thus \eqref{eq:Brab} holds.

Now, we will show weak convergence of $\mc{X}^{n}$ to $\mc{X}$.
The argument rests on applying Theorem 4.8.10 of \cite{EthierKurtz1986}, which requires us to establish that the Markov generators $\mc{Q}_n$, of $\mc{X}^{n}$, and $\mc{Q}$, of $\mc{X}$, are close, in a suitable sense.
We will denote partial derivatives of $f$ with respect to its first and second coordinate as $\frac{\p f}{\p 1}$ and $\frac{\p f}{\p 2}$ respectively.
We take the domain of $\mc{Q}_n$ to be the set of real valued continuously differentiable functions on $[0,1]^2$, and note that the generator of $\mc{X}_t$, with this same domain, is
\begin{equation}\label{eq:G_def}
\mc{Q}f(y,z)=2\zeta y(1-y)\frac{\p f}{\p 1}(y,z).
\end{equation}
Using \eqref{eq:Yj_pregen} from Lemma \ref{lem:Xnj_markov} and recalling the definition of $\mc{X}^n_t$, the generator of $\mc{X}^{n}_t$ is
\begin{align}
\mc{Q}_nf(y,z)
&=\1\l\{z\geq \tfrac{1}{2cn^\beta+2}\r\}2\zeta y(1-y)\frac{\p f}{\p 1}(y,z)\label{eq:z_uniform_0}\\
&+\1\l\{z<\tfrac{1}{2cn^\beta+2}\r\}\frac{1}{z}\Bigg(\mc{O}(z^2)\label{eq:z_uniform_1}\\
&\hspace{2pc}+(1-y)^2\l[f(y\tfrac{z'}{z},z')-f(y,z)\r]
+y(1-y)\l[f(y\tfrac{z'}{z}-z',z')-f(y,z)\r]\label{eq:Gn_def_0}\\
&\hspace{2pc}+y(1-y)\sum\limits_{r=1}^\infty\P[R=r]\sum\limits_{b=0}^r\P\l[\mc{B}_{r,1/z',y/z}=b\r]\l[f(y\tfrac{z'}{z}+z'(b-1),z')-f(y,z)\r]\notag\\
&\hspace{2pc}+y^2\sum\limits_{r=1}^\infty\P[R=r]\sum\limits_{b=0}^r\P\l[\mc{B}_{r,1/z',y/z}=b\r]\l[f(y\tfrac{z'}{z}+z'(b-2),z')-f(y,z)\r]\Bigg)
\label{eq:Gn_def}
\end{align}
Here, the $\mc{O}(z^2)$ term has subsumed an $\mc{O}(z)$ term coming from $\P[I_0=I_1=1]=\frac{y^2-yz}{1-z}=y^2+\mc{O}(z)$; noting that by Taylor's theorem
$f(y\tfrac{z'}{z}+\mc{O}(z'))-f(y,z)=\mc{O}(z),$
so that after multiplication in the above only the $y^2$ part is non-negligible.

We will show that
\begin{equation}\label{eq:gen_control_uniform}
\sup_{y,z\in[0,1]}\l|\mc{Q}_nf(y,z)-\mc{Q}f(y,z)\r|\to 0
\end{equation}
as $n\to\infty$. With this equation and Lemma \ref{lem:X_tightness} in hand it is straightforward to see that Theorem 4.8.10 of \cite{EthierKurtz1986} applies, with the desired conclusion.
We now give the proof of \eqref{eq:gen_control_uniform}. We begin by examining the two terms in \eqref{eq:Gn_def_0}. Take $y\in[0,1]$ and $z\in[0,\frac{1}{2cn^\beta+2})$. For such $z$, we have $z=\mc{O}(n^{-\beta})$ and
\begin{align}
\frac{z'-z}{z}&=\frac{2z}{1-2z}=\mc{O}(n^{-\beta}),\label{eq:z_control_1}\\
\frac{z'-z}{z^2}&=\frac{2}{1-2z}=2+\mc{O}(n^{-\beta}),\label{eq:z_control_2}
\end{align}
uniformly in such $z$.
From Taylor's theorem, and then using \eqref{eq:z_control_1} and \eqref{eq:z_control_2} we have
\begin{align}
\frac{1}{z}(1-y)^2\l[f(y\tfrac{z'}{z},z')-f(y,z)\r]
&=\frac{1}{z}(1-y)^2\l\{\frac{\p f}{\p 1}(y,z)\l[y\frac{z'}{z}-y\r]+\frac{\p f}{\p 2}(y,z)\l[z'-z\r]\r\}\notag\\
&=(1-y)^2(2y)\frac{\p f}{\p 1}(y,z)+\mc{O}(n^{-\beta})\label{eq:gen_inter_1}.
\end{align}
Similarly,
\begin{align}
\frac{1}{z}y(1-y)\l[f(y\tfrac{z'}{z}-z',z')-f(y,z)\r]
&=y(1-y)(2y-1)\frac{\p f}{\p 1}(y,z)+\mc{O}(n^{-\beta})\label{eq:gen_inter_2}.
\end{align}
Similarly again,
\begin{align}
&\frac{1}{z}y(1-y)\sum\limits_{r=1}^\infty\P[R=r]\sum\limits_{b=0}^r\P\l[\mc{B}_{r,1/z',y/z}=b\r]
\l(f(y\tfrac{z'}{z}+z'(b-1),z')-f(y,z)\r)\notag\\
=&\,\frac{1}{z}y(1-y)\sum\limits_{r=1}^\infty\P[R=r]\sum\limits_{b=0}^r\P\l[\mc{B}_{r,1/z',y/z}=b\r]
\l(\frac{\p f}{\p 1}(y,z)\l[y\frac{z'}{z}+z'(b-1)-y\r]+\frac{\p f}{\p 2}(y,z)\l[z'-z\r]\r)\notag\\
=&\,y(1-y)\frac{\p f}{\p 1}(y,z)\sum\limits_{r=1}^\infty\P[R=r]\sum\limits_{b=0}^r\P\l[\mc{B}_{r,1/z',y/z}=b\r]\l(2y+b-1\r)+\mc{O}(n^{-\beta})\notag\\
=&\,y(1-y)\frac{\p f}{\p 1}(y,z)\sum\limits_{r=1}^\infty\P[R=r]\l(2y+r(1-y)+\mc{O}(rz')-1\r)+\mc{O}(n^{-\beta})\notag\\
=&\,\,y(1-y)\big(2y+2\zeta(1-y)-1\big)\frac{\p f}{\p 1}(y,z)+\mc{O}(n^{-\beta}).\label{eq:gen_inter_3}
\end{align}
Here, to deduce the fourth line from third,
we use \eqref{eq:Brab} along with \eqref{eq:z_control_1} to give that
$\E[\mc{B}_{r,1/z',y/z}]=r(1-y)+\mc{O}(rz')+\mc{O}(n^{-\beta})$.
Then, to deduce the final line we note that $\sum_{r=1}^\infty \P[R=r]\mc{O}(rz')=\mc{O}(\E[R]z')=\mc{O}(n^{-\beta})$ and that $\zeta=\frac12\E[R]$.

Finally, using much the same calculations as in \eqref{eq:gen_inter_3} we obtain
\begin{align}
&\frac{1}{z}y^2\sum\limits_{r=1}^\infty\P[R=r]\sum\limits_{b=0}^r\P\l[\mc{B}_{r,1/z',y/z}=b\r]
\l(f(y\tfrac{z'}{z}+z'(b-2),z')-f(y,z)\r)\notag\\
=&\,\,y^2\big(2y+2\zeta(1-y)-2\big)\frac{\p f}{\p 1}(y,z)+\mc{O}(n^{-\beta}).\label{eq:gen_inter_4}
\end{align}

Putting \eqref{eq:gen_inter_1}, \eqref{eq:gen_inter_2}, \eqref{eq:gen_inter_3} and \eqref{eq:gen_inter_4} into \eqref{eq:Gn_def}, after a brief calculation (in which the terms containing $y^3$ cancel each other out) we obtain that
\begin{align}
\mc{Q}_nf(y,z)
&=\mc{Q}f(y,z)+\mc{O}(n^{-\beta})\label{eq:Gn_approx}
\end{align}
It is straightforward to check that the $\mc{O}(n^{-\beta})$ in \eqref{eq:Gn_approx} is uniform over $y,z\in[0,1]$.
This comes from the presence of the indicators in \eqref{eq:z_uniform_0} and \eqref{eq:z_uniform_1}, which ensure that
whenever $\mc{Q}_n(y,z)$ and $\mc{Q}(y,z)$ differ we must have $z\in[0,\frac{1}{2cn^\beta+2})$, and the fact that our calculations
above only contain non-negative powers of $y\in[0,1]$.
Thus, equation \eqref{eq:gen_control_uniform} follows immediately, which completes the proof.
\end{proof}

Let $t\mapsto y(t;A)$ denote the (unique) solution to \eqref{eq:stablization_ODE} subject to the condition $y(0)=A\in[0,1]$. That is,
\begin{equation}\label{eq:ytA}
y(t;A)=\frac{A}{A-(A-1)e^{-2\zeta t}}.
\end{equation}
Note that $y(t;\cdot)$ has fixed points at $A=0$ and $A=1$; the former is unstable and the latter is stable.
Given $A\in(0,1)$, the map $t\mapsto y(t;A)$ is a strictly increasing function of $t$, with $y(t;A)\to 1$ as $t\to\infty$ and $y(t;A)\to 0$ as $t\to-\infty$. Moreover, noting that $0\leq y'(t)\leq 4\zeta$, it is easily seen that for $A,B\in(0,1)$,
\begin{equation}\label{eq:y_supcont}
\lim\limits_{B\to A}\sup_{t\in\R}\l|y(t;B)-y(t;A)\r|=0.
\end{equation}
Given a function $f\in \mathscr{D}([0,1])$ we set $y^\delta_f(t)=y(t-\delta;f(0))$.

\begin{lemma}\label{lem:portmanteau_prep}
For all $\delta>0$, the following set is an open subset of $\mathscr{D}([0,1])$:
$$D_\delta=\l\{f\in\mathscr{D}([0,1])\-f(0)\in(0,1)\text{ and }\inf_{t\in[0,\infty)}f(t)-y^\delta_f(t)>0\r\}$$
\end{lemma}
\begin{proof}
Let $d_{Sk}$ denote the usual Skorohod metric on $\mathscr{D}([0,1])$, see e.g.~equation (3.5.2) of \cite{EthierKurtz1986}. We will show that
$D'_\delta=\mathscr{D}([0,1])\sc D_\delta$
is closed. Let $f_j,f\in\mathscr{D}([0,1])$ be such that $d_{Sk}(f_j,f)\to 0$ (as $j\to\infty$) and $f_j\in D'_\delta$. It remains only to show that $f\in D'_\delta$. Note that if $f(0)=0$ or $f(0)=1$ then it is automatic that $f\in D'_\delta$, so without loss of generality consider $f(0)\in(0,1)$.

Set $g_j=f_j-y^\delta_{f_j}$ and $g=f-y^\delta_f$, both elements of $\mathscr{D}([0,1])$. Then $g_j-g=(f_j-f)+(y^\delta_{f}-y^\delta_{f_j})$. It follows from $d_{Sk}(f_j,f)\to 0$ that $f_j(0)\to f(0)\in(0,1)$, and using \eqref{eq:y_supcont} we have also that $\sup_t|y^\delta_{f_j}(t)-y^\delta_f(t)|\to 0$; hence (using the definition of $d_{Sk}$) we have that $d_{Sk}(g_j,g)\to 0$.

We have $f_j\in D'_\delta$, so $\inf_{t\in[0,\infty)} g_j(t)\leq0 $.
Suppose, aiming for a contradiction, that $f\notin D_\delta$. Then $f\in D_\delta$, so there exists $\epsilon>0$ such that $\inf_{t\in[0,\infty)}g(t)\geq \epsilon$, which implies that $d_{Sk}(g_j,g)\geq\epsilon$; but this is impossible since $d_{Sk}(g_j,g)\to 0$.
Therefore we must have $f\in D_\delta$.
\end{proof}

\begin{lemma}\label{lem:ts}
Let $t=t(s)=\log\l(\frac{C}{C-s(C-c)}\r)$. Then $t$ is a bijective transformation between $[0,1]\leftrightarrow [0,\log(C/c)]$. Moreover, for any $\epsilon>0$ we have
\begin{equation}\label{eq:Tnt}
\P\l[\sup\limits_{s\in[0,1]}\l|T^n_{s(C-c)n^\beta}-t(s)\r|>\epsilon\r]\lesssim \frac{1}{\epsilon^2}\frac{1}{4cn^{\beta}}.
\end{equation}
\end{lemma}
\begin{proof}
The first claim is trivial. For the second, recall Remark \ref{rem:Cinteger}, and note that $t(s)$ is a uniformly continuous function of $s\in[0,1]$.
Thus, it suffices to prove \eqref{eq:Tnt} with the supremum over $s$ restricted to $s^n_j=\frac{j}{(C-c)n^\beta}$ for $j=0,1,\ldots,(C-c)n^\beta$. For such $j$ we have
$T^n_{s^n_j}=\sum_{l=Cn^\beta-s_j(C-c)n^\beta+1}^{Cn^\beta} E^n_l.$
Since the $(E^n_l)$ are independent it follows that
$M^n_j=T^n_{s^n_j}-\E[T^n_{s^n_j}]$
is a square integrable martingale (with parameter $j$, with respect to its generated filtration).
The maximal inequality gives $\P[\sup_j|M^n_j|>\epsilon]\leq \epsilon^{-2}\E[(M^n_J)^2]$ where $J=(C-c)n^\beta-1$. Using independence we have
\begin{align*}
\E\l[(M^n_J)^2\r]
=\sum\limits_{l=cn^\beta+1}^{Cn^\beta}\var(E_l)
=\sum\limits_{l=cn^\beta+1}^{Cn^\beta}\frac{1}{l^2}
\leq \int_{cn^\beta}^{\infty}\frac{1}{(2l)^2}\,dl
\leq \frac{1}{4cn^\beta}.
\end{align*}
Moreover,
$$\E[T^n_{s^n_j}]=\sum\limits_{l=Cn^\beta-s_j^n(C-c)n^\beta+1}^{Cn^\beta}\frac{1}{l}=\int_{Cn^\beta-s_j^n(C-c)n^\beta}^{Cn^\beta}\frac{1}{2l}\,dl+\mc{O}(n^{-\beta})$$
from which we obtain $\E[T^n_{s^n_j}]=t(s^n_j)+\mc{O}(n^{-\beta})$. The result follows.
\end{proof}

\begin{proof}[Of Proposition \ref{prop:stablization_bounds}.]
We have $C>\xi^{1/\zeta}$. We will show first that 
\begin{equation}\label{eq:Cstar2}
\P\l[Y^n_0\geq \frac{\zeta}{4C^\zeta}\r]\geq \frac{\zeta^2}{4\xi}.
\end{equation}
We have $Y^n_0=\frac{1}{2Cn^\beta+2}H^n_{Cn^\beta}$.
From the Paley-Zygmund inequality and Lemma \ref{lem:br_end_H} we thus have
\begin{align*}
\P\l[Y^n_0\geq \frac12\E[Y^n_0]\r]
\,\geq \,
\frac14\frac{(\E[H^n_{Cn^\beta}])^2}{\E[(H^n_{Cn^\beta})^2]}
\,\gtrsim\, 
\frac14\frac{\frac{4\zeta^2}{C^{2\zeta}}(1-\frac{\xi}{8\zeta C^\zeta})^2}{\xi/C^{2\zeta}}
\,\geq \,
\frac14\frac{\zeta^2(\frac12)^2}{\xi}
\,=\,
\frac{\zeta^2}{4\xi}.
\end{align*}
Note that in the final inequality above we also used that $C>\xi^{1/\zeta}$ implies $1-\frac{\xi}{8\zeta C^\zeta}\geq\frac12$.
Similarly, from Lemma \ref{lem:br_end_H} we have 
$\E[Y^n_0]\gtrsim \frac12\frac{2\zeta}{C^\zeta}(1-\frac{\xi}{4\zeta C^\zeta})\geq \frac{\zeta}{2C^\zeta}$, 
so we have that
$\P[Y^n_0\geq \frac12\E[Y^n_0]]\leq \P[Y^n_0\geq \frac{\zeta}{4C^\zeta}]$,
from which \eqref{eq:Cstar2} follows.

Equation \eqref{eq:Cstar2} is all that we know about the $Y^n_0$, so we now look to gain some `artificial' control over the initial conditions used in the limit.
To this end, independently of all else, let $(I_n)$ be a sequence of independent $\{0,1\}$ valued random variables such that
$$\P[I_n=1]=\frac{\zeta^2/4\xi}{\P\l[Y^n_0\geq \frac{\zeta}{4C^\zeta}\r]}.$$
Define a set $\mc{M}\sw\mc{U}_{Cn^\beta}$ as follows:
\begin{itemize}
\item If $\frac{1}{Cn^\beta}H^n_{Cn^\beta}\geq \frac{\zeta}{C^\zeta}$ and $I_n=1$, then let $\mc{M}$ be a uniformly random subset of $\mc{H}^n_{Cn^\beta}$ of size $\frac{\zeta}{C^{\zeta-1}}n^\beta$ (which we will assume to be an integer c.f.~Remark \ref{rem:Cinteger}).
\item Otherwise, let $\mc{M}$ be the empty set.
\end{itemize}
For $j=0,\ldots,(C-c)n^\beta$ we define
\begin{equation}\label{eq:H_hat}
\hat{\mc{H}}^n_{Cn^\beta-j}=\mc{H}^n_{Cn^\beta-j}\cap\l(\bigcup_{b\in\mc{M}}b^\downarrow\r).
\end{equation}
In words, to define $\hat{\mc{H}}^n_{Cn^\beta-j}$ we artificially remove all balls that are not within $\mc{M}$ from the genealogy at time $Cn^\beta$, and from that point onwards (looking backwards in time) we only include balls that were potential ancestors of balls in $\mc{M}$.

We define $\hat{H}^n_{Cn^\beta-j}=|\hat{\mc{H}}^n_{Cn^\beta-j}|$, define $\hat{Y}^n_j$ using \eqref{eq:Ynj_def} with $\hat{H}^n_j$ in place of $H^n_j$, define $\hat{Z}^n_j=Z^n_j$, and define $\hat{X}^n_j=(\hat{Y}^n_j,\hat{Z}^n_j)$.
It is immediate that all these quantities evolve according to the same dynamics as their counterparts without $\hat{\cdot}$s (but with different initial conditions for the first coordinate).
Moreover, \eqref{eq:H_hat} implies that $\hat{H}^n_{Cn^\beta-j}\leq H^n_{Cn^\beta-j}$ and, consequently,
\begin{equation}\label{eq:YYj}
\hat{Y}^n_j\leq Y^n_j
\end{equation}
for all $j$. From the definition of $\mc{M}$ we have that
\begin{equation*}
\hat{Y}^n_0=
\begin{cases}
\frac{\zeta}{4C^\zeta} & \text{ with probability }\frac{\zeta^2}{4\xi} \\
0 & \text{ otherwise.}
\end{cases}
\end{equation*}

Recall the processes $\mc{X}^n_t=(\mc{Y}^n_t,\mc{Z}^n_t)$ and $\mc{X}_t=(\mc{Y}_t,\mc{Z}_t)$. Take their initial states to be
\begin{align}
\mc{X}^n_0=\hat{X}^n_0&=
\begin{cases}
(\frac{\zeta}{4C^\zeta},\frac{1}{2Cn^\beta+2}) & \text{ with probability }\frac{\zeta^2}{4\xi}\\
(0,0) & \text{ otherwise}.
\end{cases} \label{eq:Xn0}\\
\mc{X}_0&=
\begin{cases}
(\frac{\zeta}{4C^\zeta},0) & \text{ with probability }\frac{\zeta^2}{4\xi}\\
(0,0) & \text{ otherwise}. \label{eq:X0}
\end{cases}
\end{align}
Since we have $\mc{X}^n_0=\hat{X}^n_0$ it follows from Lemma \ref{lem:XXcouple} that we can couple $\mc{X}^n_\cdot$ and $\hat{X}^n_\cdot$ in such a way that for all $j=0,1,\ldots,(C-c)n^\beta$, $\hat{X}^n_j=\mc{X}^n_{T^n_j}$.
Hence, in particular
\begin{equation}\label{eq:YYtj}
\hat{Y}^n_j=\mc{Y}^n_{T^n_j}.
\end{equation}
By \eqref{eq:Xn0} and \eqref{eq:X0} we have that $\mc{X}^n_0$ converges to $\mc{X}_0$. 
Hence, by Lemma \ref{lem:Xt_conv}, $\mc{X}^n_t$ converges weakly to $\mc{X}_t$ in $\mathscr{D}([0,1]^2)$, and thus $\mc{Y}^n_t$ converges weakly to $\mc{Y}_t$ in $\mathscr{D}([0,1])$.

Let $\delta>0$, to be chosen later, and recall the set $D_\delta$ from Lemma \ref{lem:portmanteau_prep}.
The path $t\mapsto \mc{Y}_t$ has dynamics given by \eqref{eq:stablization_ODE}, and with probability $\frac{\zeta^2}{4\xi}$ has initial condition within $(0,1)$; on this event we have $\mc{Y}_\cdot\in D_\delta$.
Thus, by Lemma \ref{lem:portmanteau_prep} and the Portmanteau theorem we have
\begin{equation}\label{eq:YnDdelta}
P[\mc{Y}^n_\cdot\in D_\delta] \gtrsim\frac{\zeta^2}{4\xi}.
\end{equation}

On the event $\mc{Y}^n_\cdot\in D_\delta$, we have $\mc{Y}^n_0\in(0,1)$, which by the definition of $\mc{Y}^n_0$ in \eqref{eq:Xn0} implies that $\mc{Y}^n_0=\frac{\zeta}{C^\zeta}$.
Moreover, 
on this event we have $\inf_t \mc{Y}^n_t-y(t-\delta;\frac{\zeta}{C^\zeta})>0$ and in particular,
\begin{equation}\label{eq:yTns_pre0}
\inf_t \mc{Y}^n_t-y(t-\delta;\tfrac{\zeta}{C^\zeta})>0,
\end{equation}
where $t=t(s)$ is as in Lemma \ref{lem:ts}, and $s\in[0,1]$.
By Lemma \ref{lem:ts} we have that
\begin{equation}\label{eq:Ttn_error}
\P\l[\sup_{s\in[0,1]}|T^n_{s(C-c)n^\beta}-t(s)|\leq\epsilon\r]\gtrsim 1-\frac{1}{\epsilon^2}\frac{1}{cn^\beta},
\end{equation}
where $\epsilon>0$ is to be chosen later.
Conditioning also on the event in \eqref{eq:Ttn_error}, and recalling that $t\mapsto y(t-\delta;\frac{\zeta}{C^\zeta})$ is increasing, we obtain from \eqref{eq:yTns_pre0} that, for all $s\in[0,1]$,
\begin{equation}\label{eq:yTns_pre1}
\mc{Y}^n_{T^n_{s(C-s)n^\beta}}-y(t(s)-\delta-\epsilon;\tfrac{\zeta}{C^\zeta})>0.
\end{equation}
Using \eqref{eq:ytA} and \eqref{eq:YYtj}, equation \eqref{eq:yTns_pre1} becomes
\begin{equation*}
\hat{Y}^n_{s(C-c)n^\beta}>\frac{\zeta/C^\zeta}{\zeta/C^\zeta-(\zeta/C^\zeta-1)\frac{(C-s(C-c))^{2\zeta}}{C^{2\zeta}}e^{2\zeta(\delta+\epsilon)}}
\end{equation*}
Choose $\delta=\epsilon=\frac{c}{4\zeta C}$, and after a short calculation we obtain
\begin{align}
\hat{Y}^n_{s(C-c)n^\beta}
&>\frac{\zeta}{\zeta+e^{c/C}(C-s(C-c))^{2\zeta}}.\label{eq:yTns_pre2}
\end{align}
To sum up, after accounting for the error terms incurred by conditioning on the events in \eqref{eq:YnDdelta} and \eqref{eq:Ttn_error}, we have that \eqref{eq:yTns_pre2} holds with probability $\gtrsim \frac{\zeta^2}{4\xi}$.
Proposition \ref{prop:stablization_bounds} follows immediately from this result and equation \eqref{eq:YYj}.
\end{proof}



\section{Affine preferential attachment and addition of multiple edges}
\label{sec:general_addition}

Several authors, dating at least as far back as \cite{DorogovtsevEtAl2000}, allow an extra parameter $\alpha$,
which controls the extent to which new vertices prefer to attach to existing high degree vertices.
In the classical model, the effect of $\alpha$ is that
when a new edge samples which vertex to attach to,
the existing vertices are weighted according to $\alpha+\deg_n(v)$, instead of just $\deg_n(v)$.
This mechanism is sometimes known as `affine' preferential attachment.
In PAC, we may apply the same mechanism to the sampling of potential parents.

The corresponding modification of the urn process in Section \ref{sec:urn_coupling} is that
each source ball is assigned activity $1+\alpha$, whilst cue balls have activity $1$.
Here, activity is meant in sense of (drawing balls from) generalized P\'{o}lya urns; a ball with activity $a>0$ is drawn with probability proportional to $a$.
If $\alpha$ is an integer, then at the level of the urn process this mechanism is equivalent to adding $\alpha$ new source balls,
all of colour $F_n$, on the $n^{th}$ step of the process.
For the Galton-Watson process of Section \ref{sec:gw_coupling}, the effect is that the probability of $\frac12$ for $\v{p}$ to be a source ball is replaced by $\frac{1+\alpha}{1+(1+\alpha)}$, corresponding to the idea that source balls have activity $1+\alpha$ and cue balls have activity $1$.  The offspring distribution \eqref{eq:gw_offspring} is thus modified to
\begin{equation}\label{eq:gw_offspring_alpha}
\P[M=m]=
\begin{cases}
\frac{1+\alpha}{2+\alpha} & \text{ if }m=0\\
\frac{1}{2+\alpha}\P[R=m] & \text{ if }m\in\N.
\end{cases}
\end{equation}
With these modifications, the coupling described in Section \ref{sec:gw_coupling} carries over. Now, the Galton-Watson process is supercritical when $\E[M]=\frac{1}{2+\alpha}\E[R]>1$, so the appropriate modification of Corollary \ref{cor:condensation} is that condensation now occurs if $\E[R]>2+\alpha$.
With \eqref{eq:gw_offspring_alpha} in place of \eqref{eq:gw_offspring}, equation \eqref{eq:limiting_colour} continues to hold.

An alternative, and equally natural, extension is to allow new vertices to connect to more than one existing vertex.
Models of this type are considered by, for example, \cite{BianconiBarabasi2001} and \cite{DereichOrtgiese2014}.
In PAC, we may permit each new vertex $v_n$ to connect to a random number $V_n$ of existing vertices,
with each such vertex sampled independently according to PAC mechanism.
We may also allow the sequence $(V_n)$ to be random; for simplicity we will assume it is an i.i.d.~sequence.
For the urn process associated to PAC, this means that on the $n^{th}$ step of time we would add $V_n$ new source balls, all of the same colour,
plus $V_n$ new cue balls whose colours would be inherited independently of each other using the usual mechanism.
In this case, the balance of source balls versus cue balls remains exactly even,
with the result that Corollary \ref{cor:condensation} requires no modification.
Note that, in order to obtain this result (in particular, to carry over Lemma \ref{lem:Kn_control}) we must assume that the i.i.d.~random variables $V_n$ have finite expectation.

\medskip

\noindent
\textbf{Acknowledgement.}
We thank the anonymous referee, for extensive comments that greatly improved the presentation of the article.

\appendix

\section{Appendix}
\label{app:key_asym}
\numberwithin{equation}{section}

We give a proof of Lemma \ref{lem:key_asym}.
The case $\gamma_j=0$ can be found within Exercise 8.3 of \cite{Hofstad2016}, and may be established using the Gamma function and Stirling's inequality. 
We give an elementary argument which also covers $\gamma_j\neq 0$.

Let us first consider the case in which $\gamma_j=0$ for all $j$.
In this case, Lemma \ref{lem:key_asym} is implied by the following inequality: for all $\alpha>0$ and all $k,n\in\N$ such that $2\alpha+1< k\leq n<\infty$ it holds that
\begin{equation}\label{eq:key_lemma_A}
\l(\frac{n}{k}\r)^{\alpha}\l(1-\frac{\alpha+\alpha^2}{k}\r)\,\leq\, \prod\limits_{j=k}^n\l(1+\frac{\alpha}{j}\r)\,\leq\, \l(\frac{n}{k}\r)^{\alpha}\l(1+\frac{1}{n}\r)^\alpha.
\end{equation}
The proof of \eqref{eq:key_lemma_A} proceeds as follows. We first note that
\begin{equation}\label{eq:key_lemma_B}
\prod\limits_{j=k}^n\l(1+\frac{\alpha}{j}\r)=\exp\l(\sum\limits_{j=k}^n\log\l(1+\frac{\alpha}{j}\r)\r).
\end{equation}
We use the inequality $\log(1+x)\leq x$ to obtain, for the upper bound,
\begin{align*}
\eqref{eq:key_lemma_B}
\leq \exp\l(\sum\limits_{j=k}^n\frac{\alpha}{j}\r)
\leq \exp\l(\int_k^{n+1}\frac{\alpha}{j}\,dj\r)
=\l(\frac{n}{k}\r)^\alpha\l(1+\frac{1}{n}\r)^\alpha.
\end{align*}
Here, we use that $\frac{1}{j}$ is a decreasing function of $j$, to bound the $\sum$ with an $\int$. This establishes the upper bound claimed in \eqref{eq:key_lemma_A}. For the lower bound, we use that $x-\frac{x^2}{2}\leq \log(1+x)$, to obtain
\begin{align*}
\eqref{eq:key_lemma_B}
\geq \exp\l(\sum\limits_{j=k}^n \frac{\alpha}{j}-\frac{\alpha^2}{2j^2}\r)
\geq \exp\l(\int_{k-1}^n \frac{\alpha}{j}-\frac{\alpha^2}{2j^2}\,dk\r)
=\l(\frac{n}{k}\r)^\alpha\l(1-\frac{1}{k}\r)^\alpha\l(1-\frac{\alpha^2}{2(k-1)}\r).
\end{align*}
Here, we use that $x-\frac{x^2}{2}$ is an increasing function of $x\in(0,\frac12)$, along with the assumption that $2\alpha+1< k$ to ensure that $\frac{\alpha}{k}\in(0,\frac12)$ and $k>1$. For the final line, we use that $e^{-x}\leq 1-x$ for $x\geq 0$. 
To finish we note that 
$(1-\frac{1}{k})^\alpha(1-\frac{\alpha^2}{2(k-1)})\geq (1-\frac{\alpha}{k})(1-\frac{\alpha^2}{k})\geq 1-\frac{\alpha+\alpha^2}{k}$ by \eqref{eq:taylor}.
Thus, we have established both sides of \eqref{eq:key_lemma_A}.

In the general case in which $\sum_j|\gamma_j|<\infty$, the same calculations as above results in multiplication by an additional term containing $\exp(\pm\sum_{j=k}^n|\gamma_j|)$, after the integration step. 
These terms have no effect on the asymptotic behaviour (up to first order as $k,n\to\infty$) because $\sum_{j=k}^n|\gamma_j|\to 0$ as $k,n\to\infty$.

\bibliographystyle{plainnat}
\bibliography{scalefree}

\end{document}